\documentclass[a4paper,11pt]{article}
\usepackage{amsmath,amsthm,amssymb}
\usepackage{mathrsfs}
\usepackage{amsfonts}
\usepackage[all]{xy}

\newcommand{\bgeq}{\begin{equation}}
\newcommand{\edeq}{\end{equation}}
\newcommand{\bgdisp}{\begin{displaymath}}
\newcommand{\eddisp}{\end{displaymath}}
\newcommand{\X}{\mathfrak X}

\newtheorem{definition}{Definition}[section]
\newtheorem{lemma}[definition]{Lemma}

\newtheorem{proposition}[definition]{Proposition}
\newtheorem{theorem}[definition]{Theorem}
\newtheorem{remark}[definition]{Remark}

\newtheorem{example}[definition]{Example}

\newcommand{\Tt}{\mathbb T}

\title{\textbf{Reduction and construction of Poisson quasi-Nijenhuis manifolds with background}}
\author{Fl\'avio Cordeiro\,\dag \, and \, Joana M. Nunes da Costa\,\ddag\\[0.4cm]
{\small \dag Mathematical Institute, University of Oxford,
England}\\{\small \ddag CMUC, University of Coimbra, Portugal}}

\begin{document}
\maketitle
\begin{abstract}
We extend the Falceto-Zambon version of Marsden-Ratiu Poisson
reduction to Poisson quasi-Nijenhuis structures with background on
manifolds. We define gauge transformations of Poisson
quasi-Nijenhuis structures with background, study some of their
properties and show that they are compatible with reduction
procedure. We use gauge transformations to construct Poisson
quasi-Nijenhuis structures with background.
\end{abstract}

\section*{Introduction}
Poisson quasi-Nijenhuis structures with background were recently
introduced by Antunes \cite{bib8} and include, as a particular
case,
 the Poisson quasi-Nijenhuis structures defined by Sti\'enon and Xu \cite{bib9}. The structure consists of a Poisson bivector together
 with a $(1,1)$-tensor and two closed $3$-forms fulfilling some compatibility conditions. In \cite{bib1}, Zucchini showed that some physical
 models provide a structure which is a bit more general than Poisson quasi-Nijenhuis manifolds with background. In fact, as it is observed in
  \cite{bib8}, comparing with our definition, in Zucchini's
 definition
 one condition is missing.
Generalized complex structures with background, also called
twisted generalized complex structures, are another special case
of Poisson quasi-Nijenhuis structures with background. They were
introduced by Gualtieri \cite{bib17} and further studied, among
other authors, by Lindstr\"{o}m {\em et al} \cite{bib21} and
Zucchini \cite{bib1} in relation with sigma models in physics.

In order to simplify the writing, we will use PqNb for Poisson
quasi-Nijenhuis with background, PqN for Poisson quasi-Nijenhuis,
PN for Poisson-Nijenhuis and gc for generalized complex.

The aim of this paper is two fold. Firstly, we study reduction of
PqNb manifolds and secondly, by means of a technique that we call
gauge transformation, we are able to construct these structures
from simpler ones. Moreover, we prove that these two procedures
are compatible in the sense that they commute.

One of our goals is to extend Poisson reduction to PqNb
structures. The classical Marsden-Ratiu \cite{bib6}
 method of Poisson reduction by distributions was recently reformulated by Falceto and Zambon \cite{bib24} and it is this new version of Poisson
 reduction that we apply to PqNb structures. Our scheme is the following: reduce the Poisson bivector on the manifold
 and then establish the conditions ensuring that the remaining tensor fields that define the PqNb structure also descend
 to the quotient in such a way that the reduced structure is in fact a PqNb structure.

 In this paper we view gc structures with background as particular cases of PqNb structures. Thus, in a very natural way, gc structures with background
 gain a reduction procedure which turns to be a generalization of Vaisman's reduction theorem of gc structures
 (Theorem 2.1 in \cite{bib2}).
There are other different approaches to reduction of gc structures and gc structures with background
(see \cite{bib5.5, bib4, bib5, bib3, Zambon}).

Besides reduction, the other main notion in this paper is gauge
transformation. Inspired by the corresponding notion for gc
structures, also called $B$-field transformation, we define gauge
transformations of PqNb structures and realize that they can be
seen as a tool for constructing PqNb structures from other PqNb
structures. In particular, we may construct richer examples of
such structures from simpler ones and, indeed, we construct a new
class of PqNb structures by applying gauge transformations to the
simplest PqNb structures, i.e. those consisting just of a Poisson
bivector. Unlike gauge transformations of Dirac structures which
are graphs of Poisson bivectors
  \cite{bib23, BurszRadko}, our notion
  of gauge transformation preserves the Poisson bivector of the PqNb structure. Moreover, these gauge transformations
   share very interesting
  properties, some of them we discuss, and which may be used to study the class of all PqNb structures on a given manifold. We should mention that, in \cite{bib1}, Zucchini gives a similar definition of gauge transformation with respect to the structure defined there, but he doesn't present the proof that the gauge transformations preserve such structure.

The paper is organized as follows. In section 1, PqNb structures
and gc structures with background are recalled. Section 2 is
devoted to reduction. After a brief review of Poisson reduction,
in the sense of Falceto-Zambon, we give a reduction theorem
 for PqNb manifolds and we also discuss the case of reduction by a group action. Still in section 2, we treat the
  reduction of gc structures with background. In section 3, we introduce
  the concept of gauge transformation of PqNb structures and we show how to use it to construct richer examples of PqNb structures from simpler ones. We also consider conformal change by Casimir functions and, combining it with gauge
  transformation, we obtain new examples of PqNb structures. We study some properties of gauge transformations and, finally, we show that gauge transformations commute with reduction. The paper closes with
   an appendix containing the proof of some technical lemmas.

\section{Preliminaries}
\subsection{Poisson quasi-Nijenhuis manifolds with background}\label{subsec_PqN_backgrnd}

Let us recall that  a bivector field $Q\in \X^2(M)$ on a $C^{\infty}$-differentiable manifold $M$
 determines a bracket $[\,,]_Q$ of $1$-forms $\alpha, \beta \in \Omega^1(M)$:
\bgeq \label{eq1}
[\alpha,\beta]_{Q}=\mathcal{L}_{Q^{\sharp}\alpha}(\beta) -
\mathcal{L}_{Q^{\sharp}\beta}(\alpha) - d(Q(\alpha,\beta)) \,.
\edeq

Given a $(1,1)$-tensor $A$ on $M$,
$A:TM\to TM$, $[\, ,]_A$ denotes the bracket of vector fields $X,Y \in \X^1 (M)$ deformed
by $A$: \bgdisp [X,Y]_{A}=[AX,Y] + [X,AY] -
A[X,Y] , \eddisp
 $\imath_A$ is the  $0$-degree derivation on the graded algebra $(\Omega(M),\wedge)$ given by
\begin{eqnarray}
& & (\imath_A \alpha)(X_1,X_2,\ldots,X_k) = \alpha(AX_1,X_2,\ldots,X_k) \nonumber \\
& & + \alpha(X_1,AX_2,\ldots,X_k) + \cdots +
\alpha(X_1,X_2,\ldots,AX_k)\nonumber
\end{eqnarray}
and  $d_A$ is the derivation
of degree $1$ on $(\Omega(M),\wedge)$  given by
 \bgeq\label{eq1.6}
d_A=[\imath_A,d]=\imath_A\circ d - d\circ \imath_A \,. \edeq
The {\it Nijenhuis torsion} of $A$ is the $(1,2)$-tensor
$\mathcal{N}_A$ defined by
\begin{equation}
\mathcal{N}_{A}(X,Y)  =  \left[AX,AY\right]-A(\left[AX ,Y\right] +
\left[X,AY \right] -A \left[X ,Y\right]), \, \, X,Y\in \X^1(M).\nonumber
\end{equation}

For the bracket $[\,,]_{A}$ on $\X^1(M)$, the corresponding
bracket of $1$-forms determined by the bivector $Q$, which we
denote by $([\,,]^A)_Q$, is given by \eqref{eq1} where $d$ is
replaced by $d_A$ and $\mathcal{L}$ by $\mathcal{L}^A$,
$\mathcal{L}_{X}^A=\imath_X\circ d_A + d_A\circ \imath_X \,$. The
{\it concomitant} of $P$ and $A$ is the $(2,1)$-tensor
$\mathcal{C}_{P,A}$ given by \bgdisp
\mathcal{C}_{P,A}(\alpha,\beta) =
\frac{1}{2}\left(\left([\alpha,\beta]^{A}\right)_P -
\left([\alpha,\beta]_P\right)_{A^t}\right), \eddisp where
$A^t:T^*M \to T^*M$ denotes the transpose of $A$ and $([\,,]_P )_{A^t}$ is the bracket $[\,,]_P$ deformed by
$A^t$. This is
equivalent to
\begin{eqnarray}
\mathcal{C}_{P,A}(\alpha,\beta)& = &
\mathcal{L}_{P^{\sharp}\beta}(A^t\alpha) -
\mathcal{L}_{P^{\sharp}\alpha}(A^t\beta) +
A^t\mathcal{L}_{P^{\sharp}\alpha}(\beta) -
A^t\mathcal{L}_{P^{\sharp}\beta}(\alpha) \nonumber\\
&  & + d\left(P(A^t\alpha,\beta)\right) -
A^td\left(P(\alpha,\beta)\right)\,,\label{eq1.8}
\end{eqnarray}
for all $\alpha,\beta\in\Omega^1(M)$. This concomitant
is one half of that defined in \cite{bib8}.

Poisson quasi-Nijenhuis structures with background were recently
defined by Antunes in \cite{bib8}. We now propose a slightly
different definition. For the interior product a form $\omega \in \Omega(M)$ by the bivector $X \wedge Y$, we use the convention
$\imath_{X\wedge Y}\omega=\imath_Y\imath_X \omega$.
\begin{definition}\label{def1}
A Poisson quasi-Nijenhuis structure with background on a manifold
$M$ is a quadruple $(P,A,\phi,H)$ of tensors on $M$ where $P$ is a
Poisson bivector, $A:TM\to TM$ is a $(1,1)$-tensor and $\phi$ and
$H$ are closed $3$-forms, such that
\begin{eqnarray}
 & & A\circ P^{\sharp} = P^{\sharp}\circ A^t \,,\label{eq2.1} \\
 & & \mathcal{C}_{P,A}(\alpha,\beta) = - \imath_{P^{\sharp}\alpha \wedge P^{\sharp}\beta}H \,,\label{eq2.2}\\
 & & \mathcal{N}_{A}(X,Y) = P^{\sharp}\left(\imath_{X\wedge Y}\phi + \imath_{AX\wedge Y}H + \imath_{X\wedge AY}H   \right)\,,\label{eq2.3}\\
 & & d_A\phi= d\mathcal{H} \,,\label{eq2.4}
\end{eqnarray}
for all $X,Y\in \X^1(M)$, $\alpha,\beta\in \Omega^1(M)$, and where
$\mathcal{H}$ is the $3$-form given by \bgeq\label{eq3}
\mathcal{H}(X,Y,Z) = \circlearrowleft_{X,Y,Z}H(AX,AY,Z)\,, \edeq
for all $X,Y,Z\in \X^1(M)$, the symbol $\circlearrowleft_{X,Y,Z}$
meaning a sum over the cyclic permutations of $(X,Y,Z)$. The
$3$-form $H$ is called the background and the manifold $M$ with
such structure is said to be a Poisson quasi-Nijenhuis manifold
with background.
\end{definition}
The difference between this definition and that given in
\cite{bib8} is the minus sign in equation \eqref{eq2.2}. With this
change, the definition above contains the class of
generalized complex manifolds with background
 and, moreover, enables us to define
the concept of gauge transformations of PqNb structures.

When $H=0$, this reduces to the Poisson quasi-Nijenhuis structures
defined in \cite{bib9}. If, in addition, $\phi=0$, we get the
Poisson-Nijenhuis structures introduced in
\cite{bib11}.

A very simple example of a PqNb structure is the following.

\begin{example} \label{ex1}
{\rm Consider $\mathbb R ^3$ with coordinates $(x_1,x_2,x_3)$ and
take any $C^{\infty}$ functions $f:\mathbb R ^3\to \mathbb R
\backslash \{0\}$ and $g:\mathbb R ^3\to \mathbb R $ such that
$\displaystyle{\frac{\partial g}{\partial x_1}=\frac{\partial
g}{\partial x_2}=0}$ at any point. Then, the quadruple
$(P,A,\phi,H)$ with $\displaystyle{P=f\frac{\partial}{\partial
x_1}\wedge \frac{\partial}{\partial x_2}}$, $\displaystyle{A=
g(\frac{\partial}{\partial x_1}\otimes dx_1 +
\frac{\partial}{\partial x_2}\otimes dx_2+
\frac{\partial}{\partial x_3}\otimes dx_3)}$,
$H=\displaystyle{-\frac{1}{f}\frac{\partial g}{\partial x_3}dx_1
\wedge dx_2 \wedge dx_3}$ and $\phi =-2gH$ is a PqNb structure on
$\mathbb R ^3$. Notice that $\mathcal{C}_{P,A}(\alpha,\beta) =
\frac{\partial g}{\partial x_3}P(\alpha, \beta) dx_3=-
\imath_{P^\sharp \alpha \wedge P^\sharp \beta}H$, for all
$\alpha,\beta\in\Omega^1(\mathbb R ^3)$, and
 $A$ is a Nijenhuis tensor, $\mathcal{N}_{A}(X,Y)= P^\sharp (
 \imath_{X \wedge Y}(\phi + 2 gH))=0$, for all $X, Y \in \X^1(\mathbb R
^3)$.}
\end{example}

\begin{remark} \label{remark_Zucchini}
{\rm In \cite{bib1}, Zucchini has shown that the geometry of the
Hitchin sigma model incorporates all the defining conditions of a
Poisson quasi-Nijenhuis manifold with background except the last
one, condition (\ref{eq2.4}). This was what he called an
$H$-twisted Poisson quasi-Nijenhuis manifold, which is slightly
more general than a Poisson quasi-Nijenhuis manifold with
background but does not satisfy some integrability conditions.}
\end{remark}

The concept of PqNb structure on a manifold, given in Definition
\ref{def1}, can be generalized for generic Lie algebroids. This
was in fact the approach followed in \cite{bib8} (see also
\cite{bib12}). However, in the case of the results presented here,
the generalization is always straightforward so that we prefer to
work with the standard Lie algebroid all the time.

\subsection{Generalized complex structures with background}\label{gcsb}
Let $M$ be a manifold and consider the {\it generalized tangent
  bundle} ${\mathbb T}M=TM\oplus T^*M$. This vector bundle is the ambient framework of
generalized complex geometry. This is a recent subject introduced
by Hitchin \cite{bib16}, and further studied by Gualtieri
\cite{bib17} and other authors, e.g. \cite{bib18,bib21,bib3,bib2},
which contains the symplectic and complex geometries as extreme
cases. The Lie bracket of vector fields on $M$ extends to the
well-known {\it Courant bracket} $[\,\,,\,]$ on $\Gamma({\mathbb
T}M)$: \bgeq \label{eq7} [X + \alpha, Y + \beta]= [X,Y] +
{\mathcal L}_X\beta - {\mathcal L}_Y\alpha +
\frac{1}{2}d(\alpha(Y) - \beta(X)) \,, \edeq for all $X,Y\in
\X^1(M)$ and $\alpha,\beta\in\Omega^1(M)$. Given a closed $3$-form $H$ on $M$, the
Courant bracket can be deformed into the {\it Courant bracket with background $H$},
$[\,\,,\,]_H$, by simply adding an $H$-dependent
term: \bgeq [X+\alpha,Y+\beta]_H =
[X+\alpha,Y+\beta] - \imath_{X\wedge Y}H \,. \edeq

By the well-known Newlander-Nirenberg theorem, a complex structure
on $M$ is equivalent to a vector bundle map $J:TM\to TM$
satisfying $\displaystyle{J^2=-{\rm Id}}$ and having a null
Nijenhuis torsion. The passage to the generalized case is done by
substituting the tangent bundle $TM$ by ${\mathbb T}M$ and the
bracket of vector fields by the Courant bracket:

\begin{definition} \label{def_gc}
Given a closed $3$-form $H$ on $M$, a generalized complex
structure with background $H$ on $M$ is a vector bundle map
$\mathcal{J}:{\mathbb T}M\to {\mathbb T}M$ satisfying
$\mathcal{J}^2=-{\rm Id}$ and such that the following
integrability condition holds: \bgdisp
[\mathcal{J}\mu,\mathcal{J}\nu]_H-\mathcal{J}[\mathcal{J}\mu,\nu]_H
- \mathcal{J}[\mu,\mathcal{J}\nu]_H - [\mu,\nu]_H=0 \,, \eddisp
for all $\mu,\nu\in \Gamma({\mathbb T}M)$. The triple
$(M,\mathcal{J},H)$ is called a {\it generalized
  complex manifold with background $H$} (or an $H$-twisted
generalized complex manifold). When $H=0$,
$\mathcal{J}$ is said to be a generalized complex structure and
$(M,\mathcal{J})$ a {\it generalized complex manifold}.
\end{definition}

\begin{remark}  \label{remark_def_gc}
{\rm
Equivalently, a gc structure with background $H$ on $M$ can be defined as an isotropic complex vector subbundle $L$ of the complexification $\Tt_{\mathbb C}M= \Tt M \otimes {\mathbb C}$ such that $\Tt_{\mathbb C}M=L \oplus {\overline L}$, where ${\overline L}$ is the conjugate of $L$, and $L$ is involutive with respect to the Courant bracket $[\, ,]_ H$. The complex Dirac structures $L$ and ${\overline L}$ are the $(+i)$ and $(-i)$-eigenbundles of ${\mathcal J}$. We prefer to use Definition \ref{def_gc} since it is the one that provides the natural link with PqNb structures.}
\end{remark}

The following result completely characterizes generalized complex
structures with background in terms of classical tensors. It was
referred in \cite{bib21,bib2} and it is a simple extension of the
analogous result of Crainic \cite{bib18} for the case $H=0$.
\begin{theorem}\label{theorem0.8}
A vector bundle map $\mathcal{J}:{\mathbb T}M\to {\mathbb T}M$ is
a generalized complex structure with background $H$ on $M$ if and
only if it can be written in the form \bgeq\label{eq8}
\mathcal{J}=\left(\begin{array}{cc}
A & P^\sharp \\
\sigma^\flat & -A^t \end{array}\right) \,, \edeq where $P$ is a
bivector on $M$, $\sigma$ a $2$-form on $M$, and $A:TM\to TM$ a
$(1,1)$-tensor,  such that:
\begin{itemize}
\item[{\rm(1)}]
 $P$ is a Poisson bivector;
\item[{\rm(2)}]
 $ A\circ P^\sharp = P^\sharp
\circ A^t$ \quad  and \quad
$\mathcal{C}_{P,A}(\alpha,\beta)=-\imath_{P^\sharp\alpha \wedge
  P^\sharp\beta}H$;
 \item[{\rm(3)}]
$
\mathcal{N}_A(X,Y)=P^\sharp\left(\imath_{X\wedge Y}d\sigma +
\imath_{AX\wedge Y} H + \imath_{X\wedge AY} H  \right)$;
\item[{\rm(4)}]
 the $(0,2)$-tensor $\sigma_A$ defined by
$\sigma_A(X,Y)=\sigma(AX,Y)$ is antisymmetric and satisfies the
relation \bgeq\label{eq9} d\sigma_A + H - \imath_Ad\sigma -
\mathcal{H}=0 \,, \edeq where $\mathcal{H}$ is given by
\eqref{eq3};
 \item[{\rm(5)}]
  $ A^2=-{\rm Id} -
P^{\sharp}\circ \sigma^{\flat}$.
\end{itemize}
\end{theorem}

By comparing this theorem with Definition \ref{def1}, we
immediately see that a gc structure with background is a special
case of a PqNb structure. Notice that equation (\ref{eq2.4}), with $\phi= d \sigma$, follows from (\ref{eq9}). Thus, we can write Theorem \ref{theorem0.8}
as follows:
\begin{theorem}\label{theorem0.9}
Let $\mathcal{J}:{\mathbb T}M\to {\mathbb T}M$ be a vector bundle
map of the form \eqref{eq8}, $\mathcal{J}:= (A,P, \sigma)$. Then,
$\mathcal{J}$ is a generalized complex structure with background
$H$ on $M$ if and only if $(P,A,d\sigma,H)$ is a Poisson
quasi-Nijenhuis structure with background on $M$ and properties 4
and 5 of Theorem \ref{theorem0.8} are satisfied.
\end{theorem}

Now, let $(P,A,d\sigma,H)$ be a PqNb structure on $M$ and take a
closed $2$-form $\omega$ on $M$. It is obvious that the
replacement of $\sigma$ by $\sigma + \omega$ makes no change in
the PqNb structure.  However, if, additionally,  conditions $4$
and $5$ of Theorem \ref{theorem0.8} are satisfied, i.e.
$\mathcal{J}_{\sigma}:= (A,P, \sigma)$  is a gc structure with
background $H$, one can ask under what conditions is
$\mathcal{J}_{\sigma +\omega}:= (A,P, \sigma+ \omega)$ still a gc
structure with background $H$ on $M$. An immediate computation
shows that this happens if and only if the $(0,2)$-tensor
$\omega_A$ is antisymmetric, $d\omega_A=0$ and $P^{\sharp} \circ
\omega^{\flat}=0$. This means that a PqNb structure on $M$ has
more than one gc structure with background associated with it.
Also, this defines an equivalence relation on the subclass of all
gc structures having the background $H$ and the Poisson bivector $P$.

\section{Reduction of Poisson quasi-Nijenhuis manifolds with background}
\subsection{Reduction of Poisson manifolds}
A well-known reduction procedure for Poisson manifolds is the one
due to Marsden and Ratiu \cite{bib6}. Roughly speaking, given a
Poisson manifold $(M,P)$, a submanifold $N$ of $M$ and a vector
subbundle $E$ of $TM|_N$, Marsden-Ratiu reduction theorem
establishes necessary and sufficient conditions to have a Poisson
structure on the quotient $N/(E \cap TN)$. Recently,
Falceto and Zambon \cite{bib24} showed that the assumptions of
Marsden-Ratiu theorem are too strong and they gave more flexible
hypothesis on the subbundle $E$ still providing a Poisson
structure on the quotient submanifold.

\begin{definition}\label{def1.5}
Let $(M,P)$ be a Poisson manifold, $i_N:N\subset M$ a submanifold
of $M$ and $E$ a vector subbundle of $TM|_N$ such that
 $E\cap TN$ is an integrable subbundle of $TN$ and the foliation of $N$ defined by such subbundle is simple, i.e. the set $Q$ of leaves is a manifold and the canonical projection $\pi:N \to Q$ is a submersion.
The quadruple $(M,P, N, E)$ is said to be Poisson reducible if $Q$
inherits a Poisson structure $P'$ defined by\bgeq \label{eq10}
\{f,h\}_{P'} \circ \pi = \{F,H\}_P \circ i_N \,, \edeq for any
$f,h\in C^{\infty}(Q)$ and any extensions $F,H\in C^{\infty}(M)$
of $f\circ \pi$, $h\circ \pi$, respectively, with $dF$ and $dH$
vanishing on $E$.
\end{definition}

Using the notation of \cite{bib24}, we denote by $C^\infty(M)_E$
the following subset of $C^\infty(M)$,
\[
C^\infty(M)_E:= \{ f \in C^\infty(M)\, | \, df_{| \,E}=0 \}.
\]

\begin{theorem}[\cite{bib24}] \label{theorem0.5}
Let $(M,P)$ be a Poisson manifold, $i_N:N\subset M$ a submanifold
of $M$ and $E$ a subbundle of $TM|_N$ as in Definition
\ref{def1.5}. Let $D$ be a  subbundle of $TM|_N$ such that  $E
\cap TN \subset D \subset E$ and ${\mathcal E} \subset
C^\infty(M)_E$ a multiplicative subalgebra such that the
restriction map $i_N^* : {\mathcal E} \to C^\infty(N)_{E \cap TN}$
is surjective. If
$\{ {\mathcal E}, {\mathcal E} \} \subset C^\infty(M)_D$ and
$P^\sharp(E^0)\subset TN + D$,
  then $(M,P, N, E)$ is Poisson reducible.
\end{theorem}

A special case of the theorem above, which is considered in
\cite{bib24}, occurs when $D= E \cap TN$ and ${\mathcal E} =
C^\infty(M)_E$:

\begin{proposition} \label{prop_Fal_Zam}
Let $(M,P)$ be a Poisson manifold, $i_N:N\subset M$ a submanifold
of $M$ and $E$ a subbundle of $TM|_N$ as in Definition
\ref{def1.5}. If
 $\{ C^\infty(M)_E, C^\infty(M)_E \} \subset C^\infty(M)_{E \cap TN} $ and $ P^\sharp(E^0)\subset TN $, then $(M,P, N, E)$ is Poisson reducible.
\end{proposition}

The following result, that we will use later, relates the Poisson
bivector $P$ with its reduction $P'$. Its proof is
straightforward.
\begin{lemma} \label{theorem0.75}
Let $(M,P,N,E)$ be a quadruple satisfying conditions of Proposition
\ref{prop_Fal_Zam}, so that it is Poisson reducible to  $(Q,P')$.
Then, \bgeq \label{eq11.5}
P(\widetilde{\pi^*\lambda},\widetilde{\pi^*\eta})\circ
i_N=P'(\lambda,\eta)\circ \pi  \edeq
and \bgeq \label{eq11.6}
d\pi\circ
P^{\sharp}(\widetilde{\pi^*\lambda})=P'^{\sharp}(\lambda)\circ\pi
\,, \edeq
for any $\lambda,\eta\in
\Omega^1(Q)$ and any extensions
$\widetilde{\pi^*\lambda},\widetilde{\pi^*\eta}\in \Omega^1(M)$ of
$\pi^*\lambda,\pi^*\eta$ vanishing on $E$.
\end{lemma}

A particular and important case where the assumptions of
Proposition \ref{prop_Fal_Zam} are satisfied is when a certain
canonical action (i.e. preserving the Poisson structure) of a Lie
group is given \cite{bib6}.

\begin{proposition}\label{theorem0.7}
Let $(M,P)$ be a Poisson manifold and consider a canonical action
of a Lie group $G$ on $(M,P)$ admitting an $Ad^*$-equivariant
moment map $J:M\to \mathcal{G}^*$ ($\mathcal{G}$ is the Lie
algebra of $G$ and $\mathcal{G}^*$ its dual). Suppose that $\mu\in
\mathcal{G}^*$ is a regular value of $J$ and that the isotropy
subgroup $G_\mu$ of $\mu$ for the coadjoint representation of $G$,
acts freely and properly on $N_\mu:=J^{-1}(\mu)$. Consider the
quotient $Q_\mu:=N_\mu/G_\mu$, the associated canonical projection
$\pi_\mu : N_\mu \to Q_\mu$ and the inclusion map
$i_\mu:N_\mu\subset M$ as well. Then $(Q_\mu, P_\mu)$ is a Poisson
manifold, its Poisson structure being defined by \bgdisp
\{f,h\}_{P_\mu}\circ\pi_\mu =\{F,H\}_P\circ i_\mu \,, \eddisp for
any $f,h\in C^{\infty}(Q_\mu)$ and any extensions $F,H\in
C^{\infty}(M)$ of $f\circ \pi_\mu$, $h\circ \pi_\mu$,
respectively, with $dF$ and $dH$ vanishing on $E_\mu$, where
$(E_\mu)_p=T_p(G\cdot p)$, for all $p\in N_\mu$, and $G\cdot p$ is
the orbit of $G$ containing $p$.
\end{proposition}

In this case, one proves that $(E_\mu\cap TN_\mu)_p=T_p(G_\mu\cdot
p)$, for all $p\in N_\mu$. Therefore, the distribution $E_\mu\cap
TN_\mu$ is integrable and the leaves of the foliation it
determines are the $G_\mu$-orbits in $N_\mu$. The set of leaves is
the manifold $Q_\mu$ and the canonical projection $\pi_\mu$ is a
submersion. The canonicity of the action and the fact that
$E_\mu=P^{\sharp}\left((TN_{\mu})^0\right)$ holds, ensure the two
remaining conditions of Proposition \ref{prop_Fal_Zam} are also
satisfied.

\subsection{Extension to the Poisson quasi-Nijenhuis manifolds with background}

Poisson reduction can be used as a base for reducing any manifold
of Poisson type by adding the conditions which are
 needed to reduce, in an appropriate way, the additional structure. In
particular, Marsden-Ratiu Poisson reduction was used by Vaisman in
\cite{bib14} for proving a reduction procedure for
Poisson-Nijenhuis manifolds. In \cite{bib14}, Poisson-Nijenhuis
reduction by a group action was also derived. In the sequel, we
will  extend these results to the case of Poisson quasi-Nijenhuis
manifolds with background, using the more general Falceto-Zambon
reduction procedure.

\begin{definition}\label{def1.8}
Let $(M,P,A,\phi,H)$ be a Poisson quasi-Nijenhuis manifold with
background, $i_N : N \subset M$ a submanifold of $M$, and $E$ a
vector subbundle of $TM|_N$ such that assumptions of Definition
\ref{def1.5} are satisfied. We say that $(M,P,A,\phi,H)$ is
reducible if there exists a Poisson quasi-Nijenhuis structure with
background $(P',A',\phi',H')$ on the reduced manifold $Q$, such that the tensors
$P',A',\phi',H'$ are the projections of $P,A,\phi,H$ on $Q$, i.e.
$P$ and $P'$ are related by equation (\ref{eq10}), $A$ and $A'$
are related by \bgeq \label{eq13.5} d\pi\circ A|_{TN} = A' \circ
d\pi \,, \edeq where it is assumed that $A|_{TN}$ is well defined,
i.e. that $A(TN)\subset TN$, and $\phi,\phi'$ and $H,H'$ are
related by
\begin{eqnarray}
& & i_N^*\phi = \pi^*\phi' \,,\label{eq14.1}\\
& & i_N^*H = \pi^*H' \,.\label{eq14.2}
\end{eqnarray}
\end{definition}

Next theorem gives sufficient conditions for such a reduction to
occur.
\begin{theorem}\label{theorem1}
Let $(M,P,A,\phi,H)$ be a Poisson quasi-Nijenhuis manifold with
background, $i_N:N\subset M$ a submanifold of $M$ and $E$ a
vector subbundle of $TM|_N$ as in Definition \ref{def1.5}. Assume that:
\begin{itemize}
\item[i)]$\{ C^\infty(M)_E, C^\infty(M)_E \} \subset C^\infty(M)_{E \cap TN} $; \item[ii)] $P^\sharp(E^0)\subset TN$;
\item[iii)] $A(TN)\subset TN$, $A(E)\subset E$ and $A|_{TN}$ sends
projectable vector fields to projectable vector fields; \item[iv)]
$i_N^*\left(\imath_X\phi\right)=0=i_N^*\left(\imath_X H\right)$,
for all $X\in \X^1(M)$ such that $X|_ N \in \Gamma(E)$.
\end{itemize}
Then, the tensors $P,A,\phi,H$ project to tensors $P',A',\phi',H'$
on $Q$, respectively, and $(Q,P',A',\phi',H')$ is a Poisson
quasi-Nijenhuis manifold with background.
\end{theorem}
\begin{proof}
 Lemmas \ref{theorem1.1} and \ref{theorem1.2} are needed. They are presented just after this proof.
We will prove the existence of the projections $P',A',\phi',H'$
satisfying all the conditions \eqref{eq2.1}-\eqref{eq2.4}. We will
denote by $X,Y$
 arbitrary vector fields on $N$ which are projectable
to vector fields $X'=\pi_*X$, $Y'=\pi_*Y$ on $Q$ and
$\tilde{X},\tilde{Y}\in \X^1(M)$ will be arbitrary extensions of
$X,Y$.

From Proposition \ref{prop_Fal_Zam}, we know that $P$ projects to $P'$.
As for the tensor $A$, since $A(TN)\subset TN$, we can consider
the $(1,1)$-tensor $A|_{TN}:TN\to TN$. Also, since $A|_{TN}$ sends
projectable vector fields to projectable vector fields and
$A(E\cap TN)\subset E\cap TN$, there exists a unique
$(1,1)$-tensor $A':TQ\to TQ$ satisfying $
 d\pi\circ A|_{TN} = A'\circ d\pi$.
Take now $\lambda,\eta\in \Omega^1(Q)$ and let
$\widetilde{\pi^*\lambda},\widetilde{\pi^*\eta}\in \Omega^1(M)$ be
any extensions of $\pi^*\lambda$, $\pi^*\eta$ vanishing on $E$.
Since $A(E)\subset E$, $\widetilde{\pi^*\lambda}\circ A$ and
$\widetilde{\pi^*\eta}\circ A$ are extensions of
$\pi^*(\lambda\circ A')$ and $\pi^*(\eta\circ A')$ vanishing on
$E$. Therefore, from equation \eqref{eq11.5} and the fact that $P$
and $A$ satisfy \eqref{eq2.1}, we have that \bgdisp
P'(\lambda\circ A',\eta)\circ \pi=P(\widetilde{\pi^*\lambda}\circ
A,\widetilde{\pi^*\eta})\circ
i_N=P(\widetilde{\pi^*\lambda},\widetilde{\pi^*\eta}\circ A)\circ
i_N = P'(\lambda,\eta\circ A')\circ \pi\,, \eddisp and hence $P'$
and $A'$ also satisfy \eqref{eq2.1}.

From assumption (iv) and the fact of $\phi$ and $H$ being closed,
we conclude that these forms are both projectable to closed
$3$-forms $\phi'$ and $H'$ on $Q$ defined by \eqref{eq14.1} and
\eqref{eq14.2}. Moreover, an easy computation gives
\begin{eqnarray}
& & i_N^*\left(\imath_{\tilde{X}\wedge \tilde{Y}} \phi\right)=\pi^*\left(\imath_{X'\wedge Y'} \phi'\right)  \,,\label{eq14.3}\\
& &  i_N^*\left(\imath_{\tilde{X}\wedge \tilde{Y}}
H\right)=\pi^*\left(\imath_{X'\wedge Y'} H'\right)
\,.\label{eq14.35}
\end{eqnarray}

Using Lemma \ref{theorem1.1} and equations \eqref{eq14.35} and
\eqref{eq11.6}, we get
\begin{eqnarray}
 \pi^*\left(\mathcal{C}_{P',A'}(\lambda,\eta)\right)&=& i_N^*\left(\mathcal{C}_{P,A}(\widetilde{\pi^*\lambda},\widetilde{\pi^*\eta})\right) = i_N^*\left(-\imath_{P^{\sharp}(\widetilde{\pi^*\lambda})\wedge P^{\sharp}(\widetilde{\pi^*\eta})}H \right) \nonumber\\
&=&\pi^*\left(-\imath_{P'^{\sharp}\lambda\wedge
P'^{\sharp}\eta}H'\right)\,,\nonumber
\end{eqnarray}
and since  $\pi^*$ is injective, we conclude that the concomitant of $P'$ and $A'$ satisfies \eqref{eq2.2}.

We will now compute the torsion of $A'$. From \eqref{eq13.5}, we
easily  get  \bgdisp  \mathcal{N}_{A'}(X',Y')\circ \pi=d\pi\circ
\mathcal{N}_{A|_{TN}}(X,Y) \,. \eddisp Moreover, since $di_N\circ
A|_{TN}=A\circ di_N$,  we have \bgdisp
\mathcal{N}_{A}(\tilde{X},\tilde{Y})\circ i_N=di_N\circ
\mathcal{N}_{A|_{TN}}(X,Y) \,, \eddisp and therefore
\begin{eqnarray}\label{eq14.6}
\mathcal{N}_{A'}(X',Y')\circ \pi & = & d\pi \circ \left.\left(\mathcal{N}_{A}(\tilde{X},\tilde{Y})\right)\right|_{N} \nonumber \\
             & = & d\pi \circ \left.\left(P^{\sharp}(\imath_{\tilde{X}\wedge \tilde{Y}} \phi + \imath_{A\tilde{X}\wedge \tilde{Y}}H + \imath_{\tilde{X}\wedge A\tilde{Y}}H )\right)\right|_{N} \,.
\end{eqnarray}
Notice that the vector field $P^{\sharp}(\imath_{\tilde{X}\wedge
\tilde{Y}} \phi + \imath_{A\tilde{X}\wedge \tilde{Y}}H +
\imath_{\tilde{X}\wedge A\tilde{Y}}H )$, on $M$, is tangent to
$N$. This is a direct consequence of assumptions (ii) and (iv).
Now, from \eqref{eq14.3} and \eqref{eq14.35}, and noticing that
$A\tilde{X},A\tilde{Y}$ are extensions of $A|_{TN}X,A|_{TN}Y$ and
that these last ones project to $A'X',A'Y'$, the $1$-forms in
\eqref{eq14.6} are extensions of $\pi^*\left(\imath_{X'\wedge Y'}
\phi'\right)$, $\pi^*\left(\imath_{A'X'\wedge Y'} H'\right)$,
$\pi^*\left(\imath_{X'\wedge A'Y'} H'\right)$, that vanish on $E$.
Therefore, we can use equation \eqref{eq11.6} to write
\eqref{eq14.6} as \bgdisp
\mathcal{N}_{A'}(X',Y')=P'^{\sharp}(\imath_{X'\wedge Y'} \phi' +
\imath_{A'X'\wedge Y'} H' + \imath_{X'\wedge A'Y'} H') \,, \eddisp
which is equation \eqref{eq2.3}.

It remains to check \eqref{eq2.4}. From \eqref{eq1.6} and the fact
of $\phi$ and $\phi'$ being closed, we have
$d_A\phi=-d\imath_A\phi$ and $d_{A'}\phi'=-d\imath_{A'}\phi'$, and
so, from Lemma \ref{theorem1.2}, we get
$$\pi^*( d_{A'}\phi')  =  - d(\pi^*(\imath_{A'}\phi'))= - d(i_N^*(\imath_A\phi))=i_N^*d\mathcal{H}
 = \pi^*d\mathcal{H}'.$$
This completes the proof of the theorem.
\end{proof}

\begin{lemma}\label{theorem1.1}
Let $P$ be a Poisson bivector on $M$ and $A:TM\to TM$ a
$(1,1)$-tensor. Let moreover $i_N:N\subset M$ be a submanifold of
$M$ and $E$ a vector subbundle of $TM|_N$  such that conditions i), ii) and iii)
 of  Theorem \ref{theorem1} are satisfied, so
that $P$ projects to a Poisson bivector $P'$ on $Q$ and $A$ to a
$(1,1)$-tensor $A':TQ\to TQ$. Then, \bgeq \label{eq14.8}
\pi^*\left(\mathcal{C}_{P',A'}(\lambda,\eta)\right) =
i_N^*\left(\mathcal{C}_{P,A}(\widetilde{\pi^*\lambda},\widetilde{\pi^*\eta})\right)
\,, \edeq for all $\lambda,\eta\in \Omega^1(Q)$ and any extensions
$\widetilde{\pi^*\lambda},\widetilde{\pi^*\eta}\in \Omega^1(M)$ of
$\pi^*\lambda,\pi^*\eta$ vanishing on $E$.
\end{lemma}
\begin{proof}
Take any projectable vector field $X\in \X^1(N)$ and set
$X'=\pi_*X$. Using equation \eqref{eq11.5}, we get \bgdisp
d\left(P'(\lambda,\eta)\right)(A'X')\circ\pi  =
d\left(P(\widetilde{\pi^*\lambda},\widetilde{\pi^*\eta})\right)(Adi_NX)\circ
i_N , \eddisp and \bgdisp
d\left(P'(\lambda,A'^t\eta)\right)(X')\circ\pi  =
d\left(P(\widetilde{\pi^*\lambda},A^t\widetilde{\pi^*\eta})\right)(di_NX)\circ
i_N  \,, \eddisp where, in the last equality, we used the fact
that $\pi^*(A'^t\eta)=i_N^*(A^t\widetilde{\pi^*\eta})$, which is easily seen to be equivalent
to equation \eqref{eq13.5}. Moreover, using equation
\eqref{eq11.6}, we have
\begin{eqnarray*}
 d(A'^t\lambda)(P'^{\sharp}\eta,X')\circ \pi & = & d(\pi^*(A'^t\lambda))(P^{\sharp}(\widetilde{\pi^*\eta}),X)\\
& = & d(i_N^*(A^t\widetilde{\pi^*\lambda}))(P^{\sharp}(\widetilde{\pi^*\eta}),X)  \\
& = &
d(A^t\widetilde{\pi^*\lambda})(P^{\sharp}(\widetilde{\pi^*\eta}),di_NX)\circ
i_N,
\end{eqnarray*}
and, by a similar reasoning, \bgdisp
d\eta(P'^{\sharp}\lambda,A'X')\circ\pi=d(\widetilde{\pi^*\eta})(P^{\sharp}(\widetilde{\pi^*\lambda}),Adi_NX)\circ
i_N . \eddisp  Therefore,
from \eqref{eq1.8}, we obtain
\begin{eqnarray}
&& \pi^*\left(\mathcal{C}_{P',A'}(\lambda,\eta)\right)(X)= \nonumber \\
&= & d(A'^t\lambda)(P'^{\sharp}\eta,X')\circ \pi - d(A'^t\eta)(P'^{\sharp}\lambda,X')\circ\pi + d\eta(P'^{\sharp}\lambda,A'X')\circ\pi\nonumber\\
& & - d\lambda(P'^{\sharp}\eta,A'X')\circ \pi - d\left(P'(\lambda,A'^t\eta)\right)(X')\circ\pi + d\left(P'(\lambda,\eta)\right)(A'X')\circ\pi \nonumber\\
& = &
i_N^*\left(\mathcal{C}_{P,A}(\widetilde{\pi^*\lambda},\widetilde{\pi^*\eta})\right)(X)
\nonumber \,,
\end{eqnarray}
which proves \eqref{eq14.8}.
\end{proof}

\begin{lemma}\label{theorem1.2}
Let $i_N:N\subset M$ be a submanifold of $M$, $\pi:N\to Q$ a
submersion onto a manifold $Q$, $\phi$ and $H$ closed $3$-forms on
$M$ and $\displaystyle{A:TM\to TM}$ a $(1,1)$-tensor satisfying
$A(TN)\subset TN$. Suppose that $\phi$, $H$ and $A$ are
projectable by $\pi$, i.e. there exist tensors $\phi'$, $H'$ and
$A'$ on $Q$ satisfying equations \eqref{eq13.5}, \eqref{eq14.1}
and \eqref{eq14.2}. Then,  \bgeq\label{eq15} i_N^*(\imath_{A}\phi)
= \pi^*(\imath_{A'}\phi')  \edeq and \bgeq\label{eq15.5}
i_N^*\mathcal{H} = \pi^*\mathcal{H}' \,, \edeq where $\mathcal{H}$
is given by equation \eqref{eq3} and $\mathcal{H}'$ is given by
the same equation with $H'$ and $A'$.
\end{lemma}
\begin{proof}
It is a straightforward computation.
\end{proof}
When $H=0$, Theorem \ref{theorem1} gives a reduction procedure for
Poisson quasi-Nijenhuis manifolds. If, moreover, $\phi=0$, we get
a reduction theorem for Poisson-Nijenhuis manifolds which is a
slightly more general version of the one derived in \cite{bib14}.

Now we will consider the special case of reduction by symmetries.

\begin{proposition} \label{theorem2}
Let $(M,P,G,J,\mu,N_\mu, E_\mu, Q_\mu,P_\mu)$ be as in Proposition
\ref{theorem0.7}. Let also $A:TM \to TM$ be a $(1,1)$-tensor and
$\phi$, $H$ closed $3$-forms on $M$ such that $(M,P,A,\phi,H)$ is
a Poisson quasi-Nijenhuis manifold with background and such that
the following conditions hold:
\begin{itemize}
\item[(a)]
 $dJ\circ A=dJ$ at any point of $N_\mu$;
\item[(b)] there exists an endomorphism $C$ of $\mathcal{G}$ such
that $A\tilde{\xi}=\widetilde{C\xi}$, for all $\xi\in
\mathcal{G}$, where $\tilde{\xi}$ denotes the fundamental vector
field on $M$ associated with $\xi$ by the action of $G$;
\item[(c)] $A$ is $G$-invariant, i.e. ${\mathcal L}_{\tilde{\xi}}A=0$, for
all $\xi\in \mathcal{G}$;
 \item[(d)]
$i_\mu^*(\imath_{\tilde{\xi}}\phi)=0=i_\mu^*(\imath_{\tilde{\xi}}H)$,
for all $\xi\in \mathcal{G}$.
\end{itemize}
Then, $(M,P,A,\phi,H)$ reduces to a Poisson quasi-Nijenhuis
manifold with background $(Q_\mu,P_\mu,A_\mu,\phi_\mu,H_\mu)$,
where $A_\mu$, $\phi_\mu$ and $H_\mu$ are the projections of $A$,
$\phi$ and $H$ on $Q_\mu$, respectively.
\end{proposition}
\begin{proof}
We only need to prove (iii) and (iv) of Theorem \ref{theorem1}.
 Since $T_pN_\mu={\rm ker}\,dJ(p)$,
$\forall p\in N_\mu$, condition (a) above implies that
$A(TN_\mu)\subset TN_\mu$. As for the inclusion $A(E_\mu)\subset
E_\mu$, it follows from (b) and the fact that \bgeq\label{eq15.7}
(E_\mu)_p=T_p(G\cdot p)=\{\tilde{\xi}(p):\xi \in \mathcal{G}\} \,,
\edeq for all $p\in N_\mu$. Moreover, condition (c) implies that
$A$ sends projectable vector fields to projectable vector fields,
and so condition (iii) of Theorem \ref{theorem1} holds. Finally,
that condition (d) implies condition (iv) of Theorem
\ref{theorem1} is an obvious consequence of equality
\eqref{eq15.7}.
\end{proof}
This result contains the group action reduction for
Poisson-Nijenhuis manifolds presented in \cite{bib14}, and gives
also a group action reduction for Poisson quasi-Nijenhuis
manifolds.

\subsection{Reduction of generalized complex manifolds with
  background}
Taking into account that a gc manifold with background is a
special case of a PqNb manifold, we can refine Theorem
\ref{theorem1} and construct a reduction procedure for  gc
manifolds with background as follows:
\begin{theorem}\label{theorem3}
Let $(M,\mathcal{J},H)$ be a generalized complex manifold with
 background, with $\mathcal{J}:=(A, P, \sigma)$ given by \eqref{eq8}, $i_N:N\subset M$ a submanifold of $M$ and $E$ a vector subbundle of $TM|_N$  as in Definition \ref{def1.5}. Suppose that conditions (i), (ii) and (iii) of Theorem \ref{theorem1} are satisfied and, moreover,
\begin{itemize}
 \item[(a)] $\sigma^{\flat}(TN)\subset E^0$;
\item[(b)] $i_N^*(\imath_Xd\sigma)=0=i_N^*(\imath_X H)$, for all
$X\in \X^1(M)$ such that $X|_ N \in \Gamma(E)$.
\end{itemize}
Then, the tensors $P,A,\sigma,H$ project to tensors
 $P',A',\sigma',H'$ on $Q$, respectively, and $(Q,\mathcal{J}',H')$ is a
 generalized complex manifold with background where $\mathcal{J}'$ is
 the vector bundle map determined by $P',A',\sigma'$ as in \eqref{eq8}.
\end{theorem}
\begin{proof}
By Theorem \ref{theorem0.9}, $(M,P,A,d\sigma,H)$ is a PqNb
manifold and properties (4) and (5) of Theorem \ref{theorem0.8}
hold. From Theorem \ref{theorem1} we get the
PqNb manifold $(Q,P',A',\phi',H')$, where $P',A',\phi',H'$ are the
projections of $P,A,d\sigma,H$, respectively. On the other hand,
from conditions (a) and (b) above, $\sigma$ projects to a
$2$-form $\sigma'$ on $Q$. Therefore, we have $\phi'=d\sigma'$ and
so the reduced PqNb manifold that we obtain is in fact
$(Q,P',A',d\sigma',H')$. It remains to show that the tensors $P',A',\sigma',H'$
 satisfy properties (4) and (5) of Theorem
\ref{theorem0.8}. We start by noticing that a simple computation gives \bgeq\label{eq16}
i_N^*\sigma_A=\pi^*\sigma'_{A'} \,. \edeq
Then, in particular, since $\sigma_A$ is antisymmetric,
$\sigma'_{A'}$ also is. Moreover, using  \eqref{eq16} and Lemma
\ref{theorem1.2}, we can write \bgdisp i_N^*\left(d\sigma_A + H
-\imath_Ad\sigma - \mathcal{H}
\right)=\pi^*\left(d\sigma'_{A'}+H'-\imath_{A'}d\sigma'-\mathcal{H}'
\right) \,, \eddisp and so property (4) of Theorem
\ref{theorem0.8} holds. Finally, given any projectable vector
field $X\in \X^1(N)$, we have
\begin{eqnarray}
A'^2(\pi_*X) & = & \pi_*((A|_{TN})^2X)=-\pi_*X - \pi_*(P^{\sharp}(\sigma^\flat(X)))\nonumber\\
           & = & -\pi_*X - P'^{\sharp}(\sigma'^\flat(\pi_*X)) \,,\nonumber
\end{eqnarray}
which proves (5) of Theorem \ref{theorem0.8}. In the last equality
above, we used equation \eqref{eq11.6}. In fact, $\sigma^\flat(X)$
is an extension of $\pi^*\left(\sigma'^\flat(\pi_*X) \right)$
which vanishes on $E$.
\end{proof}

When $H=0$, we recover a slightly more general version of the
reduction procedure for gc manifolds found by Vaisman in
\cite{bib2}.

Now, we will use Proposition \ref{theorem2} to construct a group
action reduction procedure for gc manifolds with background.
\begin{proposition} \label{theorem3.5}
Let $(M,P,G,J,\mu, N_\mu, Q_\mu,P_\mu)$ be as in Proposition \ref{theorem0.7}.
Let also $\sigma$ be a $2$-form, $A:TM\to TM$ a $(1,1)$-tensor and
$H$ a closed $3$-form on $M$, such that $(M,\mathcal{J},H)$ is a
generalized complex manifold with background, where the vector
bundle map $\mathcal{J}$ is  determined by $P,\sigma,A$, as in
\eqref{eq8}. Suppose that conditions (a), (b) and (c) of Proposition \ref{theorem2}
are satisfied and, moreover,
\begin{itemize}
 \item[(i)] the orbits of $G$ and the
level sets of the moment map $J$ are $\sigma$-orthogonal;
\item[(ii)]
$i_\mu^*(\imath_{\tilde{\xi}}d\sigma)=0=i_\mu^*(\imath_{\tilde{\xi}}H)$,
for all $\xi\in \mathcal{G}$. \end{itemize}
 Then, the tensors
$\sigma,A,H$ project to tensors $\sigma_\mu,A_\mu,H_\mu$ on
$Q_\mu$, respectively, and $(Q,\mathcal{J}_\mu,H_\mu)$ is a
generalized complex manifold with background, where the vector
bundle map $\mathcal{J}_\mu$ is determined by
$P_\mu,\sigma_\mu,A_\mu$, as in \eqref{eq8}.
\end{proposition}
\begin{proof}
Assumption (i) guarantees condition (a) of Theorem \ref{theorem3}. All the remaining conditions of that theorem are
satisfied (see the proof of Proposition \ref{theorem2}).
\end{proof}

\begin{remark}
{\rm There are several different approaches to reduction of
generalized complex structures (without background). In
\cite{bib3}, the reduction of a gc structure $\mathcal J$ on a
manifold $M$ is performed by the action of a Lie group $G$ on $M$.
The action should preserve $\mathcal J$ and a $G$-invariant
submanifold $N$ of $M$, where $G$ acts free and properly, is
taken. The authors obtain sufficient conditions to $\mathcal J$
descend to the quotient $N/G$. The procedure consists in reducing
the complex Dirac structures on $M$ that determine $\mathcal J$,
i.e. their $(\pm i)$-eigenbundles, to Dirac structures on $N/G$
that are going to define the reduced gc structure.

In \cite{bib5.5}, the reduction of gc structures (with background) is also based on
Dirac reduction, but with a different approach. Dirac reduction is
derived from an exact Courant algebroid reduction procedure which
involves the concept of an ``extended action" and its associated
moment map.

In \cite{bib5}, the authors introduce the concept of generalized
moment map for a compact Lie group action on a generalized complex
manifold and then use this notion to implement reduction, i.e. to
define a generalized complex structure on the reduced space. In an
appendix of the paper, this approach is extended to generalized
complex structures with background.

In \cite{Zambon}, the reduction of an exact Courant algebroid $E \to M$ is performed without any group action. Instead, a coisotropic subbundle $K \to C$ ($C$ is a submanifold of $M$) of the exact Courant algebroid $E$ is used to obtain a reduced exact Courant algebroid ${\underline E} \to {\underline C}$ . The author takes a gc structure ${\mathcal J}$ on the exact Courant algebroid $E$ and gives sufficient conditions on ${\mathcal J}$, $K$ and $C$ to ${\mathcal J}$ descend to a reduced gc structure $\underline{{\mathcal J}}$ on ${\underline E} \to {\underline C}$.
}
\end{remark}

\section{Gauge transformations of Poisson quasi-Nijenhuis structures
  with background}\label{gauge_transform}
\subsection{Definition}
An important concept in generalized complex geometry is that of
gauge transformation. As shown by Gualtieri \cite{bib17}, given a
closed $3$-form $H$ and a $2$-form $B$ on $M$, the mapping
\bgeq\label{eq20} \mathbf{B}: X + \alpha \mapsto X + \alpha +
\imath_XB \edeq is a vector bundle automorphism of ${\mathbb T}M$
which is compatible with Courant brackets with backgrounds $H$ and
$H+dB$, i.e. \bgeq\label{eq201} \mathbf{B}[X+\alpha,Y+\beta]_H =
       [\mathbf{B}(X+\alpha),\mathbf{B}(Y+\beta)]_{H+dB} \,.
\edeq

The mapping $\mathbf{B}$ is called a $B$-{\it field} or a {\it
gauge
  transformation} and its matrix representation is given by
  \begin{equation} \label{matrix_B}
  \mathbf{B}=\left(\begin{array}{cc}
{\rm Id} & 0 \\
  B^{\flat} & {\rm Id}
\end{array}\right)\,.
  \end{equation}
It acts on gc structures with background $H$ by the
invertible map $\mathcal{J}\mapsto
\mathbf{B}^{-1}\mathcal{J}\mathbf{B}$ and, as it was remarked in
\cite{bib17}, $\mathbf{B}^{-1}\mathcal{J}\mathbf{B}$ is a gc
structure with background $H+dB$. If $\mathcal{J}$ is given by
\eqref{eq8}, then \bgeq \label{eq21}
\mathbf{B}^{-1}\mathcal{J}\mathbf{B}= \left(\begin{array}{cc}
A + P^{\sharp}B^{\flat} & P^{\sharp} \\
 \sigma^{\flat}-B^{\flat}P^{\sharp}B^{\flat}-B^{\flat}A - A^tB^{\flat}  & -A^t - B^{\flat}P^{\sharp}
\end{array}\right)\,,
\edeq so that the Poisson bivector $P$ is preserved, the
$(1,1)$-tensor $A$ is replaced by $A + P^{\sharp}B^{\flat}$, and
the $2$-form $\sigma$ goes to the $2$-form $\tilde{\sigma}$ given
by \bgeq
  \tilde{\sigma}=\sigma - B_C - \imath_AB \,, \nonumber
\edeq where $C$ is the $(1,1)$-tensor $P^{\sharp}B^{\flat}$ and
$B_C$ is the $2$-form given by $B_C(X,Y)=B(CX,Y)$.

Having in mind that a gc structure with background is a special
case of a PqNb structure, we now extend the concept of gauge
transformation to the latter.
\begin{theorem}\label{theorem4.0}
Let $(P,A,\phi,H)$ be a Poisson quasi-Nijenhuis structure with
background on $M$, and $B\in \Omega^2(M)$. Consider the tensors
$\tilde{P},\tilde{A},\tilde{\phi},\tilde{H}$ on $M$ given by:
\begin{eqnarray}
 & & \tilde{P}=P \,, \label{eq23.1}\\
 & & \tilde{A}=A + P^{\sharp}B^{\flat} \,,\label{eq23.2}\\
 & & \tilde{\phi}= \phi - dB_C - d(\imath_AB)\,,\label{eq23.3}\\
 & & \tilde{H}=H+dB \,.\label{eq23.4}
\end{eqnarray}
Then, $(\tilde{P},\tilde{A},\tilde{\phi},\tilde{H})$ is a Poisson
quasi-Nijenhuis structure with background on $M$.
\end{theorem}

In order to prove the theorem, we need some lemmas. Their proofs
are included in the Appendix.

\begin{lemma}\label{theorem4.1}
Let $P$ be a Poisson bivector on $M$ and $B\in \Omega^2(M)$.
Consider the $(1,1)$-tensor $C=P^{\sharp}B^{\flat}$. Then, the
concomitant of $P$ and $C$ is given by \bgeq\label{eq24}
\mathcal{C}_{P,C}(\alpha,\beta) = -\imath_{P^{\sharp}\alpha\wedge
P^{\sharp}\beta}dB \,, \edeq for all $\alpha,\beta\in
\Omega^1(M)$, and the torsion of $C$ reads \bgeq\label{eq25}
\mathcal{N}_{C}(X,Y)= P^{\sharp}\left(\imath_{CX\wedge Y}dB +
\imath_{X\wedge CY}dB - \imath_{X\wedge Y}dB_C\right) \,, \edeq
for all $X,Y\in \X^1(M)$. Moreover, we have
\bgeq\label{eq25.2} d_CB_C=\mathcal{B}^{C,C} - dB_{C^2} \,, \edeq
where, for any $(1,1)$-tensors $S,T$, we denote \bgdisp
\mathcal{B}^{S,T}(X,Y,Z) = \circlearrowleft_{X,Y,Z} dB(SX,TY,Z)
\,, \eddisp and $B_{C^2}$ is the $2$-form defined by
$B_{C^2}(X,Y)=B(C^2X,Y)$.
\end{lemma}

\

\begin{lemma}\label{theorem4.2.5}
Take tensors $Q\in \X^2(M)$, $H\in \Omega^3(M)$ and $A\in {\rm
End}(TM)$ such that \bgeq \label{eq25.5} Q^\sharp A^t=AQ^\sharp
\,, \edeq and \bgeq \label{eq25.6} \mathcal{C}_{Q,A}(\alpha,\beta)
= -\imath_{Q^{\sharp}\alpha\wedge Q^{\sharp}\beta}H \,, \edeq for
all $\alpha,\beta\in \Omega^1(M)$. Take also $B\in\Omega^2(M)$ and
denote $C=Q^{\sharp}B^\flat$. Then, we have \bgdisp
[AX,CY]-A[CX,Y]-A[X,CY]+AC[X,Y] \eddisp \bgdisp +
[CX,AY]-C[AX,Y]-C[X,AY]+CA[X,Y] = \eddisp \bgeq\label{eq26}
Q^{\sharp}\left(\imath_{AX\wedge Y}dB + \imath_{X\wedge AY}dB -
\imath_{X\wedge Y}d(\imath_AB)+\imath_{CX\wedge Y}H +
\imath_{X\wedge CY}H \right)\,, \edeq for all $X,Y\in \X^1(M)$,
and, moreover, \bgeq \label{eq27} d_AB_C + d_C(\imath_AB) =
\mathcal{H}^{C,C} + \mathcal{B}^{A,C} + \mathcal{B}^{C,A}
 - dB_{AC} - d(\imath_{CA}B) \,,
\edeq where, for any $(1,1)$-tensors $S,T$, we denote \bgdisp
\mathcal{H}^{S,T}(X,Y,Z) = \circlearrowleft_{X,Y,Z} H(SX,TY,Z) \,,
\eddisp and $B_{AC}$ is the $2$-form defined by
$B_{AC}(X,Y)=B(ACX,Y)$.
\end{lemma}

\

\begin{lemma}\label{theorem4.3.5}
Take tensors $Q\in \X^2(M)$, $\phi,H\in \Omega^3(M)$, and $A\in
{\rm End}(TM)$, and suppose that \bgeq \label{eq27.1}
\mathcal{N}_{A}(X,Y) = Q^{\sharp}\left(\imath_{X\wedge Y}\phi +
\imath_{AX\wedge Y}H + \imath_{X\wedge AY}H \right) \,, \edeq for
all $X,Y\in \X^1(M)$. Take also $B\in\Omega^2(M)$ and denote
$C=Q^{\sharp}B^\flat$. Then, we have that \bgeq \label{eq27.2}
d_A(\imath_AB)=\mathcal{H}^{A,C} + \mathcal{H}^{C,A} +
\mathcal{B}^{A,A} - dB_{A,A} + \imath_C\phi \,, \edeq where
$B_{A,A}$ is the $2$-form given by $B_{A,A}(X,Y)=B(AX,AY)$.
\end{lemma}

\

\begin{proof}[Proof of Theorem \ref{theorem4.0}]
Let us show that $(\tilde{P},\tilde{A},\tilde{\phi},\tilde{H})$
satisfies conditions \eqref{eq2.1}-\eqref{eq2.4}. We have \bgdisp
\tilde{A}\tilde{P}^{\sharp}=(A+P^{\sharp}B^{\flat})P^{\sharp}=P^{\sharp}A^t
+ P^{\sharp} (P^{\sharp}B^{\flat})^t=\tilde{P}^{\sharp}\tilde{A}^t
, \eddisp which is \eqref{eq2.1}. Condition \eqref{eq2.2} follows
from \eqref{eq24}:
\begin{equation}\mathcal{C}_{\tilde{P},\tilde{A}}(\alpha,\beta) = \mathcal{C}_{P,A}(\alpha,\beta)+\mathcal{C}_{P,C}(\alpha,\beta)
=  - \imath_{P^{\sharp}\alpha\wedge P^{\sharp}\beta}(H+dB)
 =  - \imath_{\tilde{P}^{\sharp}\alpha\wedge
\tilde{P}^{\sharp}\beta}\tilde{H},\nonumber \end{equation}
for all $\alpha,\beta\in \Omega^1(M)$. To compute the torsion of
$\tilde{A}$ we use \eqref{eq25} and \eqref{eq26}:
\begin{eqnarray*}
\lefteqn{\mathcal{N}_{\tilde{A}}(X,Y) = \mathcal{N}_{A}(X,Y) + \mathcal{N}_{C}(X,Y) }\\
&  & + \left([AX,CY]-A[CX,Y]-A[X,CY]+AC[X,Y] \right)\\
                     & & + \left([CX,AY]-C[AX,Y]-C[X,AY]+CA[X,Y] \right)\\
                     &= &P^{\sharp}\left(\imath_{X\wedge Y}\phi + \imath_{AX\wedge Y}H + \imath_{X\wedge AY}H+\imath_{CX\wedge Y}dB + \imath_{X\wedge CY}dB - \imath_{X\wedge Y}dB_C \right)\\
                     & & + P^{\sharp}\left(\imath_{AX\wedge Y}dB + \imath_{X\wedge AY}dB - \imath_{X\wedge Y}d(\imath_AB)+\imath_{CX\wedge Y}H + \imath_{X\wedge CY}H \right)\\
                     &=& \tilde{P}^{\sharp}\left(\imath_{X\wedge Y}\tilde{\phi} + \imath_{\tilde{A}X\wedge Y}\tilde{H} + \imath_{X\wedge \tilde{A}Y}\tilde{H} \right)\,,
\end{eqnarray*}
for all $X,Y\in \X^1(M)$. Finally, from \eqref{eq25.2},
\eqref{eq27} and \eqref{eq27.2}, together with the identity $d
\circ d_A=-d_A \circ d$,
 we get
\begin{eqnarray*}
\lefteqn{d_{\tilde{A}}\tilde{\phi} =  d_A\phi - d_AdB_C - d_Ad(\imath_AB)
+ d_C\phi - d_CdB_C - d_Cd(\imath_AB) }\\
& &= d\mathcal{H}^{A,A} + d\mathcal{H}^{C,C} + d\mathcal{H}^{A,C}
+
d\mathcal{H}^{C,A}
 + d\mathcal{B}^{A,A} + d\mathcal{B}^{C,C} +
d\mathcal{B}^{A,C} + d\mathcal{B}^{C,A}\\
& &=d\tilde{\mathcal{H}},
\end{eqnarray*}
where $\tilde{\mathcal{H}}$ is the $3$-form given by $
\tilde{\mathcal{H}}(X,Y,Z)=\circlearrowleft_{X,Y,Z}\tilde{H}(\tilde{A}X,\tilde{A}Y,Z)
$, for all $X,Y,Z\in \X^1(M)$. This completes the proof.
\end{proof}

Let $\mathfrak{C}_{PqNb}(M)$ denote the class of all Poisson
quasi-Nijenhuis structures with background on $M$.

\begin{definition}\label{def2}
 Let $B$ be a $2$-form on $M$. The map
$\mathfrak{B}:\mathfrak{C}_{PqNb}(M)  \to \mathfrak{C}_{PqNb}(M) $
which assigns to each PqNb structure $(P,A,\phi,H)\in
\mathfrak{C}_{PqNb}(M)$ the PqNb structure $
(\tilde{P},\tilde{A},\tilde{\phi},\tilde{H})$ defined by equations
\eqref{eq23.1}-\eqref{eq23.4} is called the gauge transformation
on $\mathfrak{C}_{PqNb}(M)$ determined by $B$.
\end{definition}

\begin{example} \label{ex2}
{\rm Take the PqNb structure $(P,A,\phi,H)$ on $\mathbb R ^3$ of
Example \ref{ex1}. The gauge transformation of this structure
determined by the $2$-form $B=  dx_2 \wedge dx_3$ is the PqNb
structure $(\tilde{P},\tilde{A},\tilde{\phi},\tilde{H})$ with
$\tilde{P}=P$, $\tilde{A}=A+\displaystyle{ f
\frac{\partial}{\partial x_1}\otimes dx_3}$, $\tilde{\phi}=\phi$
and $\tilde{H}=H$. In this case only the $(1,1)$-tensor is
modified but $
\mathcal{C}_{\tilde{P},\tilde{A}}=\mathcal{C}_{P,A}$}.
\end{example}

\begin{remark}
{\rm
Notice that the gauge transformation of gc structures with
background, defined by \eqref{eq21}, can be recovered from Theorem \ref{theorem4.0}
if we additionally specify the
transformation of the $2$-form $\sigma$. Moreover, using the alternative definition of gc structure with background in Remark \ref{remark_def_gc}, a gauge transformation of a gc structure $L$, with background $H$, is a new gc structure $L'$ with background $H+dB$, obtained from the action of a (real) $2$-form $B$ on $L$, $L'= {\bf B} L$, with $\bf B$ given by \eqref{matrix_B}.}
\end{remark}

 Gauge transformations can preserve
 main subclasses of PqNb structures on $M$. In fact, if we
require $B$ to be closed, then the associated gauge transformation
will preserve the class of PqN structures. Moreover, given a PN
structure $(P,A)$, if $B$, besides being closed, satisfies $d(B_C
+ \imath_AB)=0$, then
$\mathfrak{B}(P,A)=(P,A+P^{\sharp}B^{\flat})$ is still a PN
structure.

\begin{example}
{\rm Consider the PN structure on $\mathbb R^3$ defined by $P=
\displaystyle{\frac{\partial}{\partial x_1} \wedge
\frac{\partial}{\partial x_2} }$ and $A= \displaystyle{e^{x_3}(
\frac{\partial}{\partial x_1} \otimes dx_1 +
\frac{\partial}{\partial x_2} \otimes dx_2 +
\frac{\partial}{\partial x_3} \otimes dx_3 + x_2
\frac{\partial}{\partial x_2} \otimes dx_3)}$ and take the
$2$-form $B=e^{x_2} dx_2 \wedge dx_3$. We have
$C=\displaystyle{e^{x_2} \frac{\partial}{\partial x_1} \otimes
dx_3}$ and $\imath _A B= 2 e^{x_3} B$. Thus, $dB=0$, $B_C=0$ and
$d(\imath _A B)=0$, and therefore, the gauge transformation of the
initial PN structure is still a PN structure.}
\end{example}

\begin{remark}\label{remark2.7}
{\rm Theorem \ref{theorem4.0} and Definition \ref{def2} can be
straightforwardly generalized
  for a generic Lie algebroid $E$ over $M$. In fact, all the computations
  made to prove Theorem \ref{theorem4.0} and Lemmas \ref{theorem4.1}-\ref{theorem4.3.5} are still valid in such
  case. So, if
  $\mathfrak{C}_{PqNb}(E)$ denotes the class of all PqNb structures on
  $E$ and a $2$-form $B$ on $E$ is given, we define the associated gauge
  transformation $\mathfrak{B}:\mathfrak{C}_{PqNb}(E)  \to
  \mathfrak{C}_{PqNb}(E)$ by setting
  $\mathfrak{B}(P,A,\phi,H)=(\tilde{P},\tilde{A},\tilde{\phi},\tilde{H})$
  where $\tilde{P}=P$, $\tilde{A}=A + C$, $\tilde{\phi}= \phi - d_EB_C - d_E(\imath_AB)$ and $ \tilde{H}=H+d_EB$, with $C=P^{\sharp}B^{\flat}$.}
\end{remark}

\begin{remark}
{\rm  It is worth to mention that the expression \emph{gauge
transformation} is used in literature, by some authors, with a
different meaning from that  in Definition \ref{def2}. We will
point out one big difference. $\bf{B}$-field operation (or gauge
transformation) defined by (\ref{eq20}) was used in \cite{bib23,
bib17, BurszRadko} to transform Dirac structures of ${\mathbb T}M$
and, due to its own properties, a gauge transformation of a Dirac
structure is still a Dirac structure (eventually with respect to a
different Courant bracket on ${\mathbb T}M$). As it is well known,
Poisson structures can be viewed as Dirac subbundles; more
precisely, if $P$ is a Poisson bivector on $M$, then its graph
$L_P$ is a Dirac structure of ${\mathbb T}M$. However, the image
of $L_P$ under the mapping (\ref{eq20}), which is a Dirac
structure, is not, in general, the graph of a Poisson bivector
\cite{bib23}. Under some mild conditions this could happen and, if
this is the case, the new Poisson tensor is different from the
initial one. The philosophy in this paper is quite different
since, according to Theorem \ref{theorem4.0}, the Poisson bivector
in a PqNb structure does not change under gauge transformations.}

\end{remark}

\subsection{Construction of Poisson quasi-Nijenhuis structures with
  background}\label{construction_PqNb}

Gauge transformations can be used as a tool for generating PqNb
structures from other PqNb structures but also to construct richer
examples from simpler ones. For example, we can construct PqNb
structures from a Poisson bivector $P$, since any Poisson
structure can be viewed as a PqNb structure where $A$, $\phi$ and
$H$ vanish. In fact, according to Definition \ref{def2}, given a
$2$-form $B$ on $M$, the associated gauge transformation takes a
Poisson structure $P$ to the PqNb structure $(P,C,-dB_C,dB)$. This
proves the following:
\begin{theorem}\label{theorem6}
Let $P$ be a Poisson bivector on $M$ and $B\in \Omega^2(M)$.
Consider the $(1,1)$-tensor $C=P^{\sharp}B^{\flat}$. Then,
$(P,C,-dB_C,dB)$ is a Poisson quasi-Nijenhuis structure with
background on $M$.
\end{theorem}
According to this theorem, we are able to construct PqNb
structures from any given $2$-form on a Poisson manifold. This
result was also derived by Antunes in \cite{bib8} using the
supergeometric techniques. Theorem \ref{theorem6} is also valid
for a generic Lie algebroid $E$ over $M$ (see Remark
\ref{remark2.7}) and this was in fact the approach followed in
\cite{bib8}.

We may now ask whether it is possible to choose a Poisson bivector
$P$ on $M$ and $B\in \Omega^2(M)$ in such a way that ${\mathcal
J}:=(C,P,-B_C)$ is a gc structure with background $dB$, i.e.
$(P,C,-dB_C,dB)$ is a PqNb structure and conditions (4) and (5) of
Theorem \ref{theorem0.8} hold.
 The answer
is no. If ${\mathcal J}$ was a gc structure with background $dB$,
then we would have $C^2=-{\rm Id}-P^{\sharp}(-B_C)^{\flat}=-{\rm
Id} + C^2$ and this is an impossible condition.

However, if (and only if) we can choose $P$ nondegenerate, there
is one, and only one, closed $2$-form $\omega$ that we can add to
$-B_C$ in order that ${\mathcal J}':= (C, P, -B_C+ \omega)$ is a
gc structure with background $dB$. This $2$-form $\omega$ is the
symplectic form associated to $P$, i.e. $\omega^{\flat} = -
(P^\sharp)^{-1}$. In fact, in this case, ${\mathcal J}'$ is the
image, by the gauge transformation determined by $B$, of the gc
structure ${\mathcal J}_{sympl}:=(0,P,\omega)$ and therefore is a
gc structure with background $dB$.

\

If the $2$-form $B$ in Theorem \ref{theorem6} additionally
satisfies \bgeq\label{eq29} \imath_{P^{\sharp}\alpha \wedge
P^{\sharp}\beta}dB=0 \,, \edeq for all $\alpha,\beta\in
\Omega^1(M)$, then it is easy to see that the contribution of the
background $dB$ in equations \eqref{eq2.2}-\eqref{eq2.4} vanishes,
i.e., \bgdisp \mathcal{C}_{P,C}(\alpha,\beta)=0 \quad,\quad
\mathcal{N}_{C}(X,Y) = P^{\sharp}\left(\imath_{X\wedge
Y}(-dB_C)\right) \quad,\quad  d_C(-dB_C)=0 \,. \eddisp We have
therefore the following result:
\begin{theorem}\label{theorem7}
Let $P$ be a Poisson bivector on $M$ and $B$ a $2$-form on $M$
satisfying \eqref{eq29}. Then, $(P,C,-dB_C)$ is a Poisson
quasi-Nijenhuis structure on $M$. In particular, this is true when
$dB=0$.
\end{theorem}
This theorem gives a way of constructing PqN structures from a
$2$-form on a Poisson manifold. Moreover, in its version for a
generic Lie algebroid $E$ over $M$, this result contains Theorem
3.2 in \cite{bib22}. Just notice that, when \bgeq \label{eq30}
\imath_{P^{\sharp}\alpha}d_EB=0 \,, \edeq for all $\alpha\in
\Gamma(E^*)$, we have $\imath_C d_EB=0$ and therefore
$$[B,B]_P=2\imath_C d_EB -2d_EB_C=-2d_EB_C.$$ So, when \eqref{eq30} holds
and $[B,B]_P=0$, the pair $(P,C)$ is a PN structure on $E$. Notice
also that we do not need the anchor of $E$ to be injective as it
was required in \cite{bib22}.

So far, we have used gauge transformations to construct PqNb and
PqN structures from simpler ones. But we can also use gauge
transformations in the opposite way,
  i.e. to get simpler structures from richer ones. For example, given
  a PqNb structure $(P,A,\phi,H)$ on $M$ with $H$ exact, we can choose
  $B\in \Omega^2(M)$ such that $H=dB$ and then consider the gauge
  transformation associated with $-B$, which takes $(P,A,\phi,H)$ to
  the PqN structure $(P,A-C,\phi-dB_C+d(\imath_AB))$. By imposing
  additional restrictions  on $B$, we may obtain a PN structure or even a
  Poisson one. Also, by considering gauge transformations associated
  with closed $2$-forms, we are able to turn PqN structures into PN or
  even Poisson structures.

Next, we will show that more examples of PqNb structures can be
constructed if we combine conformal change with gauge
transformation. First, notice that if $P$ is a Poisson bivector on
$M$ and $f \in C^\infty(M)$ is a Casimir of $P$, then $e^f P$ is a
Poisson tensor:
$$ [e^f P, e^f P ]= e^f ( 2\, [P, e^f ] \wedge P + e^f
[P,P])=0.$$ The bivector $P'=e^f P$ is called the {\em conformal
change} of $P$ by $e^f$.

Take a Poisson bivector $P$ on $M$, a Casimir $f \in C^\infty(M)$
of $P$ and a $2$-form $B$ on $M$. According to Theorem
\ref{theorem6}, $(P,C,-dB_C,dB)$, with $C= P^\sharp B^ \flat$, is
a PqNb structure on $M$, which is obtained from the Poisson tensor
$P$ by the gauge transformation determined by $B$. Consider now
the Poisson bivector $P'=e^fP$ and the $2$-form $B'=e^{-f} B$.
Applying again Theorem \ref{theorem6}, we get a new PqNb structure
on $M$, $(P',C',-dB'_{C'},dB')$, which is related to $(P,C,-dB_C
,dB)$ by the formulae:
\begin{equation}
 P' = e^f P; \, C' = C; \, dB'_{C'} = e^{-f}(dB_C -df \wedge B_C); \, dB'=  e^{-f}(dB-df \wedge B). \nonumber
\end{equation}
We see that the $(1,1)$-tensor $C$ is fixed, while all the other
tensors change. However, if we wish, we may fix the background of
the PqNb structure. It suffices to apply Theorem \ref{theorem6} to
the Poisson tensor $P'=e^f P$ and the $2$-form $B'=B$. In this
case, the $(1,1)$-tensor $C= P^\sharp B ^\flat$ changes to $C'=e^f
C$ and $dB'_{C'}= e^f (dB_C+ df \wedge B_C)$.

Summarizing, we have proved the following:

\begin{proposition}
Let $P$ be a Poisson bivector on $M$, $f \in C^\infty(M)$  a
Casimir of $P$ and $B$ a $2$-form on $M$. Consider the
$(1,1)$-tensor $C= P^\sharp B^ \flat$. Then, $$(e^f P,\, C,\,
e^{-f}(-dB_C + df \wedge B_C), \, e^{-f}(dB - df \wedge B))$$ and
$$(e^f P,\, e^f C,\, e^{f}(-dB_C - df \wedge B_C),\, dB)$$ are Poisson
quasi-Nijenhuis structures with background on $M$.
\end{proposition}

\subsection{Some properties of gauge transformations}

Let us now consider the set $Gauge(M)$ of all gauge
transformations on $\mathfrak{C}_{PqNb}(M)$ and denote by
$\mathfrak{G} :\Omega^2(M) \to Gauge(M)$ the map which assigns to
each $2$-form $B$ on $M$ the gauge transformation $\mathfrak{B}$
associated with $B$, i.e. $\mathfrak{B}=\mathfrak{G}(B)$. We can
give $Gauge(M)$ a natural group structure as follows:
\begin{theorem} \label{gauge:group}
The set $Gauge(M)$ is an abelian group under the composition of
maps, the identity element being the gauge transformation
associated with the zero $2$-form, and the inverse of
$\mathfrak{G}(B)$ being $\mathfrak{G}(-B)$, for all $B\in
\Omega^2(M)$. Moreover, the map $\mathfrak{G}$ is a group
isomorphism from the abelian group $(\Omega^2(M), +)$ into
$(Gauge(M), \circ)$.
\end{theorem}
\begin{proof}
Given any $B_1,B_2\in \Omega^2(M)$, the composition of the
associated gauge transformations is given by
$(\mathfrak{G}(B_1)\circ \mathfrak{G}(B_2))(P,A,\phi,H) =
(\hat{P},\hat{A},\hat{\phi},\hat{H})$ where
\begin{eqnarray}
& & \hat{P} = P  \nonumber  \\
& & \hat{A} = A + C_1 + C_2  \nonumber \\
& & \hat{\phi} = \phi - dB_{2C_2} -
  d(\imath_AB_2) - dB_{1C_1} - d(\imath_AB_1)  - d(\imath_{C_2}B_1)
   \label{eq27.7}
  \\
& & \hat{H} = H + dB_1 + dB_2  \nonumber
\end{eqnarray}
with $C_i=P^{\sharp}B_i^{\flat}$, $i=1,2$. Since
 $B_{2}(C_1X,Y)=B_1(X,C_2Y)$, for all $X,Y\in \X^1(M)$,
we can write \eqref{eq27.7} as \bgdisp \hat{\phi} = \phi -
d(B_1+B_2)_{C_1 + C_2} - d\imath_A(B_1+B_2) \,,  \eddisp and so we
realize that the composition $\mathfrak{G}(B_1)\circ
\mathfrak{G}(B_2)$ is indeed the gauge transformation associated
with $B_1 + B_2$, i.e. \bgdisp
 \mathfrak{G}(B_1+B_2) = \mathfrak{G}(B_1)\circ
\mathfrak{G}(B_2)\,. \eddisp From this relation, the proof of the
first part of the theorem is obvious and this same relation means
that $\mathfrak{G}$ is a group homomorphism. Since by definition
$\mathfrak{G}$ is a surjection, it just remains to prove that it
is an injection. Take $B\in \Omega^2(M)$ and suppose that
$\mathfrak{G}(B) = {\rm Id}$. Then, applying $\mathfrak{G}(B)$ on
PqNb structures of the form $(P,0,0,0)$, we see that
$P^{\sharp}B^{\flat} = 0$ for all Poisson bivectors $P$ on $M$.
Therefore, we must have $B = 0$. In fact, for any point $m \in M$,
we can find local coordinates around $m$, and a bump function on
$M$ which is nonzero at $m$, and prove that if $B\neq 0$, we can
construct a Poisson tensor $P$ such that $P^{\sharp}B^{\flat} \neq
0$.
\end{proof}
We conclude, from Theorem \ref{gauge:group}, that there exists a
group action of $\Omega^2(M)$ on  $\mathfrak{C}_{PqNb}(M)$, given
by
\[
\begin{array}{ccl}
\Omega^2(M) \times \mathfrak{C}_{PqNb}(M) & \to &
\mathfrak{C}_{PqNb}(M) \\
(B, (P,A, \phi,H)) & \mapsto & (P,A+ C, \phi- d B_C- d (\imath_A
B), H+dB).
\end{array}
\]

Two elements of $\mathfrak{C}_{PqNb}(M)$ are said to be {\em gauge
equivalent} if they lie in the same orbit. All equivalent PqNb
structures on $M$ have the same Poisson tensor. However, from the
results of the previous section, one single orbit may contain
different types of structures, i.e we can have gauge equivalence
between Poisson and PqNb structures, between PqN and PqNb, and so
on. In the case of a nondegenerate Poisson bivector we derive the
following:

\begin{proposition}
Given a nondegenerate Poisson bivector $P$ on $M$, the set of all
PqNb structures having $P$ as the associated Poisson bivector is
the $\Omega^2(M)$-orbit of the Poisson structure $(P,0,0,0)$. In
other words, these PqNb structures are all those of the form
$(P,P^{\sharp}B^{\flat},-dB_{P^{\sharp}B^{\flat}},dB)$ with $B\in
\Omega^2(M)$.
\end{proposition}
\begin{proof}
Let $(P,A,\phi,H)$ be a PqNb structure where $P$ is nondegenerate.
Because $P^{\sharp}$ is invertible, $\phi$ and $H$ are the unique
$3$-forms satisfying equations \eqref{eq2.2} and \eqref{eq2.3} for
$P$ and $A$. On the other hand, as a consequence of equation
\eqref{eq2.1}, the $(0,2)$-tensor $B$ defined by
$B^{\flat}=(P^{\sharp})^{-1}A$ is antisymmetric and therefore we
can write $A=P^{\sharp}B^{\flat}$ with $B\in \Omega^2(M)$.
Moreover, the gauge transformation associated with $B$ of the
Poisson structure $(P,0,0,0)$ is
$(P,P^{\sharp}B^{\flat},-dB_{P^{\sharp}B^{\flat}},dB)$. Therefore,
since the $3$-forms $\phi$ and $H$ are unique, we must have
$\phi=-dB_{P^{\sharp}B^{\flat}}$ and $H=dB$. This proves the
result.
\end{proof}

In particular, we have seen that, given a nondegenerate Poisson
bivector $P$ and a $(1,1)$-tensor $A$ satisfying equation
\eqref{eq2.1}, we have one and only one PqNb structure of the form
$(P,A,\cdot,\cdot)$. For degenerate Poisson bivectors, this is not
in general true. For example, given a degenerate Poisson bivector
$P$ on $M$ and $B\in \Omega^2(M)$,
$(P,P^{\sharp}B^{\flat},-dB_{P^{\sharp}B^{\flat}},dB)$ is a PqNb
structure and $(P,P^{\sharp}B^{\flat},-dB_{P^{\sharp}B^{\flat}},dB
+ H)$ is a PqNb structure as well, where $H$ is any $3$-form
satisfying \bgeq\label{eq50} \imath_{P^{\sharp}\alpha \wedge
P^{\sharp}\beta}H=0 \,, \edeq for all $\alpha,\beta \in
\Omega^1(M)$. In particular, when $dB$ satisfies equation
\eqref{eq50}, then
$(P,P^{\sharp}B^{\flat},-dB_{P^{\sharp}B^{\flat}},dB)$ and
$(P,P^{\sharp}B^{\flat},-dB_{P^{\sharp}B^{\flat}},0)$ are both
PqNb structures. It may also happen that, given a degenerate
Poisson bivector $P$ and a $(1,1)$-tensor $A$ such that equation
\eqref{eq2.1} holds, does not exist any PqNb structure
associated with $P$ and $A$ at all. For example, if we take $P$ to
be the null bivector and $A$ any non-Nijenhuis tensor, then
equation \eqref{eq2.1} is trivially satisfied but equation
\eqref{eq2.3} can never hold.

\subsection{Compatibility with reduction}
Now we will consider the concepts of gauge transformation and
reduction of PqNb structures and prove that they commute.
\begin{theorem}\label{theorem5}
Let $(M,P,A,\phi,H)$ be a Poisson quasi-Nijenhuis manifold with
background, $i_N:N\subset M$ a submanifold and $E$ a vector
subbundle of $TM|_N$ as
in Definition \ref{def1.5}, and suppose that all the conditions
of Theorem \ref{theorem1} are satisfied, so that $(M,P,A,\phi,H)$
is reducible  to a Poisson quasi-Nijenhuis manifold with
background $(Q,P',A',\phi',H')$. Let also $B$ be a $2$-form on $M$
such that:
\begin{itemize}
\item[(a)]
 $B^{\flat}(TN)\subset E^0$ \,;
\item[(b)] $B$ is projectable to a $2$-form $B'$ on $Q$.
\end{itemize}
 Consider the gauge transformation of $(P,A,\phi,H)$
associated with $B$,
$(\tilde{P},\tilde{A},\tilde{\phi},\tilde{H})$, as in Theorem
\ref{theorem4.0}. Then,
$(M,\tilde{P},\tilde{A},\tilde{\phi},\tilde{H})$ reduces  to a
Poisson quasi-Nijenhuis manifold with background
$(Q,\tilde{P}',\tilde{A}',\tilde{\phi}',\tilde{H}')$ which is also
the gauge transformation of $(P',A',\phi',H')$ associated with
$B'$. In other words, the diagram \bgdisp
\xymatrix{ (M,P,A,\phi,H) \ar[r]^{\mathfrak{B}} \ar[d]^{\pi} & (M,\tilde{P},\tilde{A},\tilde{\phi},\tilde{H}) \ar[d]^{\pi} \\
           (Q,P',A',\phi',H') \ar[r]^{\mathfrak{B}'} &  (Q,\tilde{P}',\tilde{A}',\tilde{\phi}',\tilde{H}')}
\eddisp is commutative, where $\mathfrak{B},\mathfrak{B}'$ are the
gauge transformations on $\mathfrak{C}_{PqNb}(M)$,
$\mathfrak{C}_{PqNb}(Q)$ associated with $B,B'$, respectively, and
$\pi:N\to Q$ is the canonical projection.
\end{theorem}
\begin{proof}
The gauge transformation of
$(P',A',\phi',H')$, associated with $B'$, is the PqNb structure
$(\tilde{P'},\tilde{A'},\tilde{\phi'},\tilde{H'})$ on $Q$ where
$\tilde{P'}=P'$, $\tilde{A'}=A'+C'$,
$\tilde{\phi'}=\phi'-dB'_{C'}-d(\imath_{A'}B')$, and
$\tilde{H'}=H'+dB'$, with $C'$ denoting the $(1,1)$-tensor
$P'^{\sharp}B'^{\flat}$. Therefore, by Definition \ref{def1.8},
we have to prove that the tensors
$\tilde{P},\tilde{A},\tilde{\phi},\tilde{H}$, given by equations
\eqref{eq23.1}-\eqref{eq23.4}, project respectively to the tensors
$\tilde{P'},\tilde{A'},\tilde{\phi'},\tilde{H'}$. By assumption,
$P,A,\phi,H,dB$ project to $P',A',\phi',H',dB'$, respectively.
Moreover, $C$ projects to $C'$ (see the proof of Theorem 3.1 in
\cite{bib14}; condition (a) above, as well as condition (ii) in
Theorem \ref{theorem1}, are needed here) and this implies that
$B_C$ projects to $B'_{C'}$:
\begin{eqnarray}
(\pi^*B'_{C'})(X,Y) & = & B'(C'\pi_*X,\pi_*Y)\circ \pi = (\pi^*B')(C|_{TN}X,Y) \nonumber\\
   & = & (i_N^*B)(C|_{TN}X,Y) = (i_N^*B_C)(X,Y)\,, \nonumber
\end{eqnarray}
for all projectable vector fields $X,Y\in \X^1(N)$, so that in
particular $dB_C$ projects to $dB'_{C'}$. With a similar
reasoning, we prove that $\imath_{A}B$ projects to $\imath_{A'}B'$
and consequently we have the same for their exterior derivatives.
\end{proof}

\section*{Appendix}
{\bf Proof of Lemma \ref{theorem4.1}.}
For equations \eqref{eq24} and \eqref{eq25} see reference
\cite{bib11}, formulas (B.3.9) and (B.3.8), respectively. As for
equation \eqref{eq25.2}, we have
\begin{eqnarray*}
\lefteqn{d_CB_C(X,Y,Z) = (CX)B(CY,Z) - (CY )B(CX,Z) + (CZ)B(CX,Y)}\\
&- & B(C[CX,Y],Z) - B(C[X,CY],Z) + B(C^2[X,Y],Z) \\
&+ & B(C[CX,Z],Y) + B(C[X,CZ], Y ) - B(C^2[X,Z],Y) \\
&- & B(C[CY,Z],X) - B(C[Y,CZ],X) + B(C^2[Y,Z],X)  \\
& = & dB(CX,CY,Z) + B([CX,CY],Z) + dB(CY,CZ,X) -(CY)B(CZ,X)  \\
& +& (CZ)B(CY,X) + B([CY,CZ],X) + dB(CZ,CX,Y) + (CX)B(CZ,Y) \\
&+ & B([CZ,CX],Y) - dB_{C^2}(X,Y,Z)\\
&= &  \mathcal{B}^{C,C}(X,Y,Z) - dB_{C^2}(X,Y,Z) - P([B^{\flat}X,B^{\flat} Y]_P,B^{\flat} Z)  \\
& -&  P([B^{\flat} Y,B^{\flat} Z]_P,B^{\flat} X) - P([B^{\flat}Z,B^{\flat} X]_P,B^{\flat} Y) \\
& + & P^{\sharp}(B^{\flat} Y)P(B^{\flat} Z,B^{\flat} X) - P^{\sharp}(B^{\flat} Z)P(B^{\flat} Y,B^{\flat} X)- P^{\sharp}(B^{\flat} X)P(B^{\flat} Z,B^{\flat} Y)\\
& =& \mathcal{B}^{C,C}(X,Y,Z) - dB_{C^2}(X,Y,Z)  + d_PP(B^{\flat} X,B^{\flat} Y,B^{\flat} Z)\\
& =& \mathcal{B}^{C,C}(X,Y,Z) - dB_{C^2}(X,Y,Z) \,,
\end{eqnarray*}
for all $X,Y,Z\in \X^1(M)$, where we used the fact of $P$ being
Poisson in the third and in the last equalities.
{\hfill $\Box$}

\

\noindent {\bf Proof of Lemma \ref{theorem4.2.5}.}
For proving \eqref{eq26}, we take $\alpha\in \Omega^1(M)$ and
apply it on the right hand side (RHS) of the equation. This gives,
using \eqref{eq25.5}, \eqref{eq25.6} and \eqref{eq1.8},
\begin{eqnarray*}
\lefteqn{\alpha({\rm RHS}) = -dB(AX,Y,Q^\sharp\alpha) - dB(X,AY,Q^\sharp\alpha)+ d(\imath_AB)(X,Y,Q^\sharp\alpha)}\\
& - & H(CX,Y,Q^{\sharp}\alpha) - H(X,CY,Q^{\sharp}\alpha) \\
& = & -(AX)B(Y,Q^{\sharp}\alpha) + B([AX,Y],Q^{\sharp}\alpha) -
  B([AX,Q^{\sharp}\alpha],Y)\\
  & + &  (AY)B(X,Q^{\sharp}\alpha) + B([X,AY],Q^{\sharp}\alpha) +
  B([AY,Q^{\sharp}\alpha],X) \\
& + &  XB(Y,AQ^{\sharp}\alpha) - YB(X,AQ^{\sharp}\alpha) - B(A[X,Y],Q^{\sharp}\alpha)- B([X,Y],AQ^{\sharp}\alpha)  \\
& + &   B(A[X,Q^{\sharp}\alpha],Y) - B(A[Y,Q^{\sharp}\alpha],X)-\mathcal{C}_{Q,A}(B^{\flat}X,\alpha)(Y) +\mathcal{C}_{Q,A}(B^{\flat}Y,\alpha)(X) \\
& = & B([AX,Y],Q^{\sharp}\alpha) + B([X,AY],Q^{\sharp}\alpha) -
  B(A[X,Y],Q^{\sharp}\alpha)- B([X,Y],AQ^{\sharp}\alpha)  \\
  & - &  \alpha(A[CX,Y]) +
\alpha([CX,AY]) + \alpha(A[CY,X]) - \alpha([CY,AX]) \\
& = & \alpha\left([AX,CY]-A[CX,Y]-A[X,CY]+AC[X,Y] \right.  \\
 & + & \left.  [CX,AY]-C[AX,Y]-C[X,AY]+CA[X,Y]\right) \,,
\end{eqnarray*}
for any $X,Y\in \X^1(M)$. As for \eqref{eq27}, we have
\begin{eqnarray*}
\lefteqn{d_AB_C(X,Y,Z) + d_C(\imath_AB)(X,Y,Z) = (AX)B(CY,Z)- (AY)B(CX,Z)}  \\
&+ &  (AZ)B(CX,Y) + (CX)(\imath_AB)(Y,Z) - (CY)(\imath_AB)(X,Z)  \\
&+ & (CZ)(\imath_AB)(X,Y) - B(C[X,Y]_A,Z) + B(C[X,Z]_A,Y) - B(C[Y,Z]_A,X)   \\
& -& (\imath_AB)([X,Y]_C,Z) +(\imath_AB)([X,Z]_C,Y) - (\imath_AB)([Y,Z]_C,X)    \\
& = & \mathcal{B}^{A,C}(X,Y,Z) + \mathcal{B}^{C,A}(X,Y,Z) - dB_{AC}(X,Y,Z) - d(\imath_{CA}B)(X,Y,Z)   \\
 & - & C_{Q,A}(B^{\flat}X,B^{\flat}Y)(Z) - C_{Q,A}(B^{\flat}Y,B^{\flat}Z)(X) - C_{Q,A}(B^{\flat}Z,B^{\flat}X)(Y)  \\
& = & \mathcal{B}^{A,C}(X,Y,Z) + \mathcal{B}^{C,A}(X,Y,Z) -
 dB_{AC}(X,Y,Z) - d(\imath_{CA}B)(X,Y,Z)   \\
 & +&  \mathcal{H}^{C,C}(X,Y,Z) \,,
\end{eqnarray*}
for all $X,Y,Z\in \X^1(M)$.
{\hfill $\Box$}

\

\noindent {\bf Proof of Lemma \ref{theorem4.3.5}.}
One just has to expand $d_A(\imath_AB)$ and then use
\eqref{eq27.1}:
\begin{eqnarray*}
\lefteqn{d_A(\imath_AB)(X,Y,Z)=(AX)(\imath_AB)(Y,Z) -
(AY)(\imath_AB)(X,Z)+ (AZ)(\imath_AB)(X,Y)
 }\\
 & &- (\imath_AB)([X,Y]_A,Z) + (\imath_AB)([X,Z]_A,Y)
- (\imath_AB)([Y,Z]_A,X) \\
&=& \mathcal{B}^{A,A}(X,Y,Z) - dB_{A,A}(X,Y,Z) - B^{\flat}(Z)\left(\mathcal{N}_{A}(X,Y)\right) \nonumber\\
& & + B^{\flat}(Y)\left(\mathcal{N}_{A}(X,Z)\right) - B^{\flat}(X)\left(\mathcal{N}_{A}(Y,Z)\right) \nonumber\\
&=& \mathcal{B}^{A,A}(X,Y,Z) - dB_{A,A}(X,Y,Z) \nonumber\\
& & - B^{\flat}(Z)\left(Q^{\sharp}\left(\imath_{X\wedge Y}\phi +
\imath_{AX\wedge Y}H + \imath_{X\wedge AY}H \right)\right) \\
& & + B^{\flat}(Y)\left(Q^{\sharp}\left(\imath_{X\wedge Z}\phi +
\imath_{AX\wedge Z}H + \imath_{X\wedge AZ}H \right)\right) \\
& & - B^{\flat}(X)\left(Q^{\sharp}\left(\imath_{Y\wedge Z}\phi +
\imath_{AY\wedge Z}H + \imath_{Y\wedge AZ}H \right)\right) \\
&=& \mathcal{B}^{A,A}(X,Y,Z) - dB_{A,A}(X,Y,Z) + (\imath_C\phi)(X,Y,Z)\\
& &  + \mathcal{H}^{A,C}(X,Y,Z) + \mathcal{H}^{C,A}(X,Y,Z)
,
\end{eqnarray*}
for all $X,Y,Z\in \X^1(M)$.
{\hfill $\Box$}

\section*{Acknowledgements}
This work has been partially supported by FCT grant SFRH/BD/44594/2008 (FC) and by CMUC-FCT and project PTDC/MAT/69635/2006 (JMNdC).


\begin{thebibliography}{25}

\bibitem{bib8} Antunes P., {\it Poisson quasi-Nijenhuis structures with
background}, Lett. Math. Phys., {\bf 86} (2008) 33-45.

\bibitem{BurszRadko} Bursztyn H. and Radko O., {\it Gauge equivalence of Dirac structures and symplectic
groupoids}, Ann. Inst. Fourier, {\bf 53} (1) (2003) 309-337.

\bibitem{bib5.5} Bursztyn H., Cavalcanti G.R. and Gualtieri M., {\it Reduction of Courant algebroids and generalized complex structures}, Adv. Math., {\bf 211} (2) (2007) 726-765.

\bibitem{bib12} Caseiro R., De Nicola A. and Nunes da Costa J.M., {\it On Poisson quasi-Nijenhuis Lie algebroids}, arXiv:0806.2467v1 [math.DG].

\bibitem{bib18} Crainic M., {\it Generalized complex structures and Lie brackets}, arXiv:math/0412097v2 [math.DG].

\bibitem{bib24} Falceto F. and Zambon M., {\it An extension of the Marsden-Ratiu reduction for Poisson manifolds},
 Lett. Math. Phys., {\bf 85} (2008) 203-219.

\bibitem{bib17} Gualtieri M., {\it Generalized Complex Geometry}, arXiv:math/0401221v1 [math.DG].

\bibitem{bib16} Hitchin N., {\it Generalized Calabi-Yau manifolds}, Quart. Journal Math., {\bf 54} (3) (2003) 281-308.

\bibitem{bib4} Hu S., {\it Hamiltonian symmetries and reduction in generalized geometry}, Houston J. Math., {\bf 35} (3) (2009) 787-811.

\bibitem{bib5} Lin Y. and Tolman S., {\it Symmetries in generalized K\"{a}hler geometry}, Comm. Math. Phys., {\bf 268} (1)
(2006) 199-222.

\bibitem{bib21} Lindstr\"{o}m U., Minasian R., Tomasiello A. and Zabzine M., {\it Generalized complex manifolds and supersymmetry}, Comm. Math. Phys., {\bf 257} (1) (2005) 235-256.

\bibitem{bib11} Magri F. and Morosi C., {\it A geometrical characterization of integrable Hamiltonian systems through the theory of Poisson-Nijenhuis manifolds}, Quaderno S 19, Univ. of Milan, 1984.

\bibitem{bib6} Marsden J. and Ratiu T., {\it Reduction of Poisson Manifolds}, Lett. Math. Phys., {\bf 11} (1986)
161-169.

\bibitem{bib23} $\check{\rm S}$evera P. and Weinstein A., {\it Poisson geometry with a 3-form background}, Prog. Theor. Phys. Suppl., no 144 (2001) 145-154.

\bibitem{bib9} Sti\'enon M. and Xu P., {\it Poisson Quasi-Nijenhuis Manifolds}, Comm. Math. Phys., {\bf 270} (2007)
709-725.

\bibitem{bib3} Sti\'enon M. and Xu P., {\it Reduction of generalized complex structures}, J. Geom. Phys., {\bf 58} (2008)
105-121.

\bibitem{bib22} Vaisman I., {\it Complementary $2$-forms of Poisson structures}, Compositio Mathematica, {\bf 101} (1) (1996) 55-75.

\bibitem{bib14} Vaisman I., {\it Reduction of Poisson-Nijenhuis manifolds}, J. Geom. Phys., {\bf 19} (1996)
90-98.

\bibitem{bib2} Vaisman I., {\it Reduction and submanifolds of generalized complex manifolds}, Differential Geom. Appl., {\bf 25} (2007) 147-166.

\bibitem{Zambon} Zambon M., {\it Reduction of branes in generalized complex geometry}, J. Symplectic Geom., {\bf 6} (4) (2008) 353-378.

\bibitem{bib1} Zucchini R., {\it The Hitchin model, Poisson-quasi-Nijenhuis geometry and symmetry reduction}, J. High Energy Phys., no 10 (2007) 075.


\end{thebibliography}
\end{document}